\documentclass[times,final]{elsarticle}

\usepackage{framed,multirow}

\usepackage{amssymb}
\usepackage{latexsym}

\usepackage{url}
\usepackage{xcolor}
\definecolor{newcolor}{rgb}{.8,.349,.1}

\usepackage{amsmath}
\usepackage{amsfonts}
\usepackage{subeqnarray}
\usepackage{subfigure}
\usepackage{epsfig}
\usepackage{bigcenter}
\usepackage{cancel}
\usepackage{lineno}

\def\du{[\![}
\def\df{]\!]}
\def\ol{\overline}

\def\phib{\boldsymbol{\phi}}
\def\etab{\boldsymbol{\eta}}
\def\ie{\textit{i.e.}}

\newtheorem{lemma}{Lemma}[section]
\newtheorem{corollary}{Corollary}[section]

\newtheorem{theorem}{Theorem}[section]
\newdefinition{remark}{Remark}[section]
\newproof{proof}{Proof}

\begin{document}


\begin{frontmatter}

\title{Entropy stable DGSEM for nonlinear hyperbolic systems in nonconservative form with application to two-phase flows}

\author[rvt1]{Florent Renac\corref{cor1}}
\ead{florent.renac@onera.fr}
\cortext[cor1]{Corresponding author. Tel.: +33 1 46 73 37 44; fax.: +33 1 46 73 41 66.}
\address[rvt1]{DAAA, ONERA, Universit\'e Paris Saclay, F-92322 Ch\^atillon, France}


\begin{abstract}
In this work, we consider the discretization of nonlinear hyperbolic systems in nonconservative form with the high-order discontinuous Galerkin spectral element method (DGSEM) based on collocation of quadrature and interpolation points (Kopriva and Gassner, J. Sci. Comput., 44 (2010), pp.136--155; Carpenter et al., SIAM J. Sci. Comput., 36 (2014), pp.~B835-B867). We present a general framework for the design of such schemes that satisfy a semi-discrete entropy inequality for a given convex entropy function at any approximation order. The framework is closely related to the one introduced for conservation laws by Chen and Shu (J. Comput. Phys., 345 (2017), pp.~427--461) and relies on the modification of the integral over discretization elements where we replace the physical fluxes by entropy conservative numerical fluxes from Castro et al. (SIAM J. Numer. Anal., 51 (2013), pp.~1371--1391), while entropy stable numerical fluxes are used at element interfaces. Time discretization is performed with strong-stability preserving Runge-Kutta schemes. We use this framework for the discretization of two systems in one space-dimension: a $2\times2$ system with a nonconservative product associated to a linearly-degenerate field for which the DGSEM fails to capture the physically relevant solution, and the isentropic Baer-Nunziato model. For the latter, we derive conditions on the numerical parameters of the discrete scheme to further keep positivity of the partial densities and a maximum principle on the void fractions. Numerical experiments support the conclusions of the present analysis and highlight stability and robustness of the present schemes.
\end{abstract}


\end{frontmatter}


%
%
\section{Introduction}

The discussion in this paper focuses on the high-order discretization of the Cauchy problem for nonlinear hyperbolic systems in nonconservative form:

\begin{subeqnarray}\label{eq:NL_model_pb}
 \partial_t {\bf u} + {\bf A}({\bf u})\partial_x{\bf u} &=& 0,\quad \mbox{in }\mathbb{R}\times(0,\infty),\\
 {\bf u}(\cdot,0) &=& {\bf u}_{0}(\cdot),\quad\mbox{in }\mathbb{R},
\end{subeqnarray}

\noindent where ${\bf u}(x,t)$  represents the vector of unknowns with values in the set of states $\Omega^a\subset\mathbb{R}^m$ and ${\bf A}:\Omega^a\ni{\bf u}\mapsto{\bf A}({\bf u})\in\mathbb{R}^m\times\mathbb{R}^m$ is a smooth matrix-valued function. We assume that system (\ref{eq:NL_model_pb}a) is strictly hyperbolic over the set of states. When there exists a flux function ${\bf f}:\Omega^a\rightarrow\mathbb{R}^m$ such that ${\bf A}({\bf u})={\bf f}'({\bf u})$ for all ${\bf u}$ in $\Omega^a$, (\ref{eq:NL_model_pb}a) can be written in conservative form for which the concept of weak solutions in the sense of distributions is used to define admissible solutions.

In the general case where ${\bf A}$ is not the Jacobian of a flux function, the theory of distributions do not apply which makes difficult to give a meaning to the nonconservative product ${\bf A}({\bf u})\partial_x{\bf u}$ at a point of discontinuity of the solution. The work by Dal Maso, Lefloch, and Murat \citep{dalmaso_etal_95} generalizes the notion of weak solutions from conservation laws to (\ref{eq:NL_model_pb}) and allows to define the nonconservative product for functions of bounded variations by extending the definition by Volpert \cite{volpert_67}. The definition is based on a family of Lipschitz paths $\phib:[0,1]\times\Omega^a\times\Omega^a\rightarrow\Omega^a$ satisfying the following properties:

\begin{equation}\label{eq:DLM_path}
 \phib(0;{\bf u}^-,{\bf u}^+)={\bf u}^-, \quad \phib(1;{\bf u}^-,{\bf u}^+)={\bf u}^+, \quad \phib(s;{\bf u},{\bf u})={\bf u}.
\end{equation}

We refer to \citep{dalmaso_etal_95} for the complete theory and requirements on the associated paths. Across a discontinuity of speed $\sigma$, the nonconservative product ${\bf A}({\bf u})\partial_x{\bf u}$ is then defined as the unique Borel measure defined by the so-called generalized Rankine-Hugoniot condition

\begin{equation}\label{eq:generalized_RH}
 \sigma\du{\bf u}\df = \int_0^1{\bf A}\big(\phib(s;{\bf u}^-,{\bf u}^+)\big)\partial_s\phib(s;{\bf u}^-,{\bf u}^+)ds,
\end{equation}

\noindent where $\du{\bf u}\df = {\bf u}^+-{\bf u}^-$, ${\bf u}^-$ and ${\bf u}^+$ are the left and right limits of ${\bf u}$ across the discontinuity. Note that the notion of weak solutions now depends on the family of paths in (\ref{eq:generalized_RH}) under consideration \cite{lefloch_89}.

Admissible weak solutions have to satisfy an entropy inequality 

\begin{equation}\label{eq:entropy_ineq_cont}
 \partial_t\eta({\bf u}) + \partial_xq({\bf u}) \leq 0,
\end{equation}

\noindent for the smooth entropy-entropy flux pair $(\eta,q)$ with $\eta(\cdot)$ a strictly convex function such that $\etab'({\bf u})^\top{\bf A}({\bf u})={\bf q}'({\bf u})^\top$ for all ${\bf u}$ in $\Omega^a$. In practice, it may be useful to also consider PDEs with both conservative and nonconservative terms because they require different approaches for their discretizations:

\begin{equation}\label{eq:PDE_cons_noncons}
 \partial_t{\bf u} + \partial_x{\bf f}({\bf u}) + {\bf c}({\bf u})\partial_x{\bf u} = 0,
\end{equation}

\noindent so for smooth solution we have ${\bf A}\equiv{\bf f}'+{\bf c}$ and the entropy pair satisfies $\etab'({\bf u})^\top\big({\bf f}'({\bf u})+{\bf c}({\bf u})\big)={\bf q}'({\bf u})^\top$ for ${\bf u}$ in $\Omega^a$.

The objective of this work is to develop a general method to design arbitrary high-order schemes for (\ref{eq:NL_model_pb}) that satisfy the entropy inequality (\ref{eq:entropy_ineq_cont}) at the semi-discrete level. We propose to use the discontinuous Galerkin spectral element method (DGSEM) based on the collocation between interpolation and quadrature points defined from Gauss-Lobatto quadrature rules \cite{kopriva_gassner10}. Using diagonal norm summation-by-parts (SBP) operators and the entropy conservative numerical fluxes from Tadmor \cite{tadmor87}, a semi-discrete entropy conservative DGSEM has been derived in \cite{carpenter_etal14}. The particular form of the SBP operators allows to take into account  the numerical quadratures that approximate integrals in the numerical scheme compared to other techniques that require their exact evaluation to satisfy the entropy inequality \cite{jiang_shu94,hiltebrand_mishra_14}. The work in \cite{chen_shu_17} provides a general framework for the design of entropy conservative and entropy stable DGSEM for the discretization of nonlinear systems of conservation laws. Numerical experiments highlight the benefits on stability and robustness of the computations, though this not guaranties to preserve neither the entropy stability at the discrete level, nor positivity of the numerical solution which is necessary to define the entropy. Designs of fully discrete entropy stable and positive DGSEM have been proposed in \cite{despres98,despres00,renac17a,renac17b}. A general framework for the design of entropy conservative and entropy stable schemes on simplex elements for sready-state conservation laws has been recently proposed in \cite{abgrall_18} that encompasses residual distribution schemes, discontinous and continuous Galerkin methods whith general quadrature formulas.

Some works rely on the discontinuous Galerkin approximation of nonconservative systems in the fields of either the shallow water flows \cite{gassner_etal_17,dumbser_casulli_13,tassi_etal_08}, or magnetohydrodynamics (MHD) \cite{liu_etal_18,dumbser_casulli_13}, or two-phase flows \cite{zwieten_etal_17,franquet_perrier12,franquet_perrier13,rhebergen_etal_08,h_de_frahan_etal_15,tokareva_toro_10,fraysse_etal_16,dumbser_casulli_13}, etc. Note that the works in \cite{gassner_etal_17} and \cite{liu_etal_18} use the DGSEM as discretization method and derive, respectively, high-order entropy conservative and well balanced discretization of the shallow water equations through skew-symmetric splitting techniques, and entropy stable schemes for the ideal compressible MHD equations by using two-point numerical fluxes from \cite{chandrashekar_klingenberg_16} at element interfaces and treating the nonconservative product as source terms without particular treatment. Though not exhaustive, we also refer to the works in \cite{abgrall_kumar_14,dumbser_boscheri_13,dumbser_etal_10,dumbser_loubere_16,tokareva_toro_16} and references therein as alternative techniques for high-order approximations of two-phase flows.

Here, we extend the work in \cite{chen_shu_17} to nonconservative products by using the two-point entropy conservative numerical fluxes in fluctuation form introduced in \cite{castro_etal_13}. This extension is clarified through the direct link between fluctuation fluxes and conservative fluxes in the case of conservation laws. The difficulty in the design of an entropy stable DGSEM lies in the treatment of the integrals over discretization elements which contain space derivatives of test functions whose sign cannot be controlled. The use of entropy conservative numerical fluxes in those integrals allows however to remove their contribution to the global entropy production in the element. The properties of high-order accuracy and approximation of the cell averaged numerical solution are more difficult to derive due to the specific form of the fluctuation fluxes. Indeed, the consistency condition has less physical meaning for fluctuation fluxes compared to conservation fluxes which require homogeneity properties in closed form. Moreover, even in the case of path-conservative fluxes \cite{pares_06} they require a priori knowledge of the underlying path. We thus introduce some assumptions on the form of the entropy conservative fluctuation fluxes and derive conditions on the scheme to keep high-order accuracy and a same semi-discrete scheme for the cell averaged approximate solution as in the original DGSEM. The method is fairly general and we provide examples of entropy conservative fluxes for nonconservative systems in various fields such as spray dynamics, gas dynamics, or two-phase flows. A deeper analysis is given for the discretization of two two-phase flow models in one space-dimension: a $2\times2$ system with a nonconservative product associated to a linearly-degenerate (LD) characteristic field, and the isentropic Baer-Nunziato model. We provide a numerical example where the original DGSEM applied to the former model is shown to fail to capture the entropy weak solution. The use of an entropy stable DGSEM scheme is here necessary to capture the correct solution and improve robustness of the computations. For the latter model, we further analyze the properties of the discrete scheme and derive conditions on the time step to keep positivity of the partial densities and a maximum principle on the void fractions. These properties hold for the cell averaged numerical solution and motivate the use of a posteriori limiters \cite{zhang_shu_10a,zhang_shu_10b} to extend them to nodal values within elements. Again, numerical experiments highlight stability and robustness improvement with the entropy stable scheme.

The paper is organized as follows. Section \ref{sec:DG_discr} presents the DGSEM for the space discretization of nonconservative systems (\ref{eq:NL_model_pb}) and its entropy stable version through the use of entropy conservative fluxes. In section \ref{sec:scheme_properties}, we derive the semi-discrete entropy inequality and give conditions on the numerical fluxes to keep high-order accuracy and the semi-discrete scheme for the cell averaged numerical solution. Various examples of entropy conservative fluxes are given in section~\ref{sec:examples} for different nonconservative systems. We further investigate the stability and robustness properties of an entropy stable DGSEM for the isentropic Baer-Nunziato model in section~\ref{sec:BN}. Numerical experiments with application to two-phase flows are given in section~\ref{sec:num_xp}. Finally, concluding remarks about this work are given in section \ref{sec:conclusions}.

%
%
\section{DGSEM formulation}\label{sec:DG_discr}

The DG method consists in defining a semi-discrete weak formulation of problem (\ref{eq:NL_model_pb}). The domain is discretized with a grid $\Omega_h=\cup_{j\in\mathbb{Z}}\kappa_j$ with cells $\kappa_j = [x_{j-\frac{1}{2}}, x_{j+\frac{1}{2}}]$, $x_{j+\frac{1}{2}}=jh$ and $h>0$ the space step (see Figure~\ref{fig:stencil_1D}) that we assume to be uniform without loss of generality.

\subsection{Numerical solution} We look for approximate solutions in the function space of discontinuous polynomials ${\cal V}_h^p=\{v_h\in L^2(\Omega_h):\;v_h|_{\kappa_{j}}\in{\cal P}_p(\kappa_{j}),\; \kappa_j\in\Omega_h\}$, where ${\cal P}_p(\kappa_{j})$ denotes the space of polynomials of degree at most $p$ in the element $\kappa_{j}$. The approximate solution to (\ref{eq:NL_model_pb}) is sought under the form

\begin{equation}\label{eq:num_sol}
 {\bf u}_h(x,t)=\sum_{l=0}^{p}\phi_j^l(x){\bf U}_j^{l}(t), \quad \forall x\in\kappa_{j},\, \kappa_j\in\Omega_h,\, t\geq0,
\end{equation}
\noindent where ${\bf U}_j^{0\leq l \leq p}$ constitute the degrees of freedom (DOFs) in the element $\kappa_j$. The subset $(\phi_j^0,\dots,\phi_j^{p})$ constitutes a basis of ${\cal V}_h^p$ restricted onto a given element. In this work we will use the Lagrange interpolation polynomials $\ell_{0\leq k\leq p}$ associated to the Gauss-Lobatto nodes over the segment $[-1,1]$: $s_0=-1<s_1<\dots<s_p=1$:

\begin{equation}\label{eq:lagrange_poly}
 \ell_k(s_l)=\delta_{k,l}, \quad 0\leq k,l \leq p,
\end{equation}
\noindent with $\delta_{k,l}$ the Kronecker symbol. The basis functions with support in a given element $\kappa_j$ thus write $\phi_j^k(x)=\ell_k(\sigma_j(x))$ where $\sigma_j(x)=2(x-x_j)/h$ and $x_j=(x_{j+\frac{1}{2}}+x_{j-\frac{1}{2}})/2$ denotes the center of the element.

The DOFs thus correspond to the point values of the solution: given $0\leq k\leq p$, $j$ in $\mathbb{Z}$, and $t\geq0$, we have ${\bf u}_h(x_j^k,t)={\bf U}_j^k(t)$ for $x_j^k=x_j+s_kh/2$. The left and right traces of the numerical solution at interfaces $x_{j\pm\frac{1}{2}}$ of a given element hence read (see Figure~\ref{fig:stencil_1D}):

\begin{subeqnarray}\label{eq:LR_traces}
 {\bf u}_{j+\frac{1}{2}}^-(t) &:=& {\bf u}_h(x_{j+\frac{1}{2}}^-,t) = {\bf U}_j^p(t), \quad \forall t\geq0,\\
 {\bf u}_{j-\frac{1}{2}}^+(t) &:=& {\bf u}_h(x_{j-\frac{1}{2}}^+,t) = {\bf U}_j^0(t), \quad \forall t\geq0.
\end{subeqnarray}

\begin{figure}
\begin{center}
\epsfig{figure=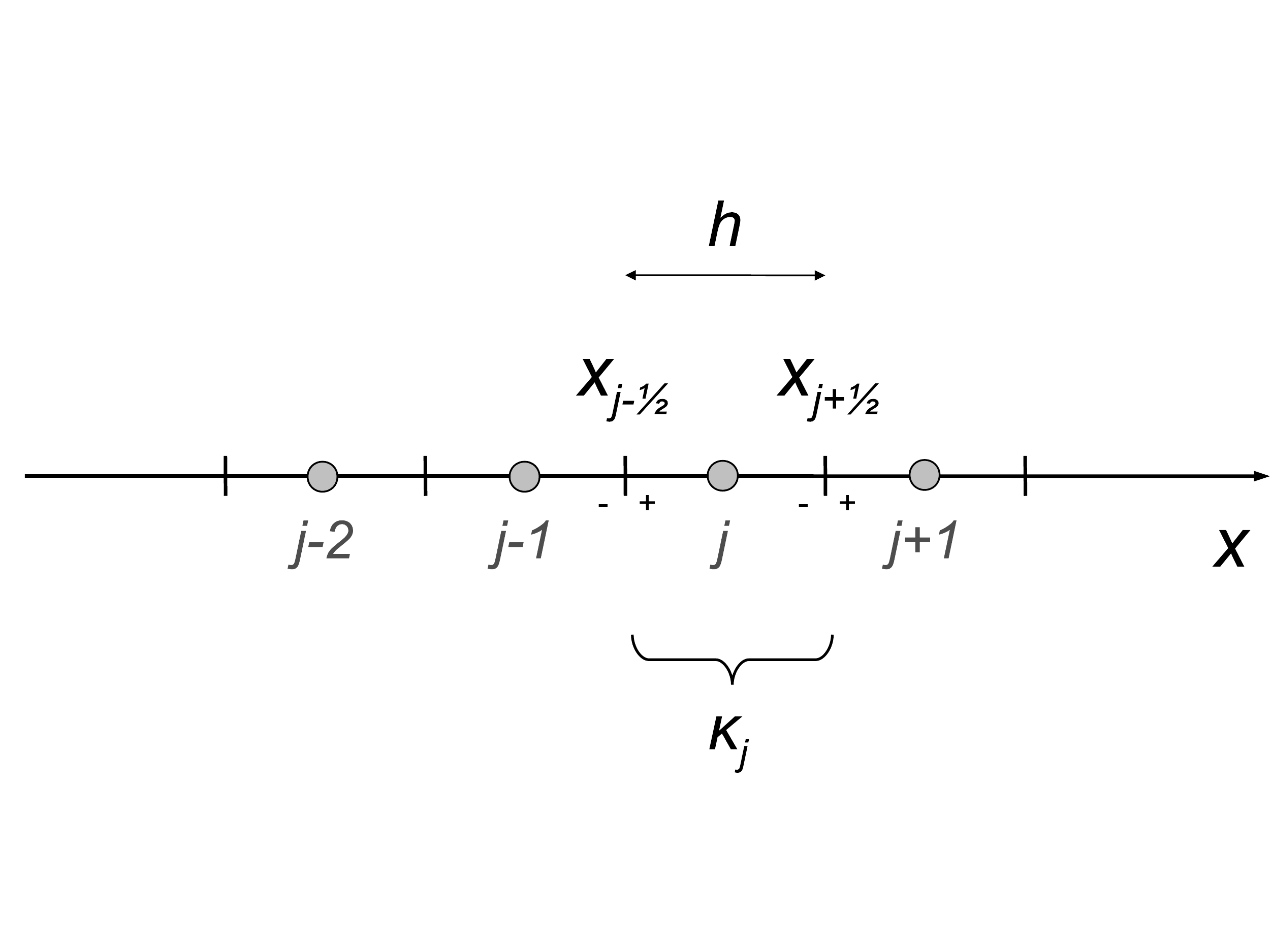,width=6cm,clip=true,trim=0cm 3.5cm 0cm 3.cm}
\caption{Mesh with definition of left and right traces at interfaces $x_{j\pm\frac{1}{2}}$.}
\label{fig:stencil_1D}
\end{center}
\end{figure}

It is convenient to introduce the difference matrix with entries

\begin{equation}\label{eq:nodalGL_diff_matrix}
 D_{kl} = \ell_l'(s_k) = \frac{h}{2}d_x\phi_j^l(x_j^k), \quad 0\leq k,l \leq p.
\end{equation}

In the DGSEM, the integrals over elements are approximated by using a Gauss-Lobatto quadrature rule with nodes collocated with the interpolation points of the numerical solution

\begin{equation}\label{eq:GaussLobatto_quad}
 \int_{\kappa_j}f(x)dx \simeq \frac{h}{2} \sum_{l=0}^p \omega_l f(x_j^l),
\end{equation}

\noindent with $\omega_l>0$, $\sum_{l=0}^p\omega_l=2$, $x_j^l=x_j+s_lh/2$ the weights and nodes of the quadrature rule, and $s_l$ defined in (\ref{eq:lagrange_poly}). This leads to the definition of the discrete inner product in the element $\kappa_j$

\begin{equation*}
 \langle f,g\rangle_j^p := \frac{h}{2}\sum_{l=0}^p\omega_lf(x_j^l)g(x_j^l).
\end{equation*}

As noticed in \cite{kopriva_gassner10}, the DGSEM satisfies the summation-by-parts property:

\begin{equation}\label{eq:SBP}
 \omega_kD_{kl}+\omega_lD_{lk}=\delta_{kl}(\delta_{kp}-\delta_{k0}), \quad 0\leq k,l \leq p.
\end{equation}

Note also that the property $\sum_{k=0}^p\ell_k\equiv1$ implies

\begin{equation}\label{eq:interp_lag_unite_deriv}
 \sum_{l=0}^pD_{kl} = 0, \quad 0\leq k\leq p.
\end{equation}

\subsection{Space discretization}

The semi-discrete form of the DG discretization in space of problem (\ref{eq:NL_model_pb}) reads \cite{franquet_perrier12,rhebergen_etal_08}: find ${\bf u}_h$ in $({\cal V}_h^p)^m$ such that

\begin{eqnarray}\label{eq:semi-discr_var_form}
 \int_{\Omega_h} v_h \partial_t{\bf u_h} dx &+& \sum_{\kappa_j\in\Omega_h}\int_{\kappa_j}v_h{\bf A}({\bf u}_h)\partial_x{\bf u}_hdx\nonumber\\
 &+& \sum_{j\in\mathbb{Z}}v_{j+\frac{1}{2}}^-{\bf D}^-\big({\bf u}_{j+\frac{1}{2}}^-(t),{\bf u}_{j+\frac{1}{2}}^+(t)\big)\nonumber\\ 
 &+& \sum_{j\in\mathbb{Z}}v_{j-\frac{1}{2}}^+{\bf D}^+\big({\bf u}_{j-\frac{1}{2}}^-(t),{\bf u}_{j-\frac{1}{2}}^+(t)\big) = 0, \quad \forall v_h\in{\cal V}_h^p, t>0,
\end{eqnarray}
\noindent where the numerical fluxes ${\bf D}^\pm(\cdot,\cdot)$ in fluctuation form will be defined below.

The projection of the initial condition (\ref{eq:NL_model_pb}b) onto $({\cal V}_h^p)^m$ reads

\begin{equation*}
 \int_{\Omega_h} v_h(x) {\bf u_h}(x,0) dx = \int_{\Omega_h} v_h(x) {\bf u}_0(x) dx, \quad \forall v_h\in{\cal V}_h^p.
\end{equation*}

Substituting $v_h$ for the Lagrange interpolation polynomials (\ref{eq:lagrange_poly}) and using the Gauss-Lobatto quadrature (\ref{eq:GaussLobatto_quad}) to approximate the volume integrals, (\ref{eq:semi-discr_var_form}) becomes

\begin{equation}\label{eq:semi-discr_DGSEM}
 \frac{\omega_kh}{2}\frac{d{\bf U}_j^k}{dt}+\omega_k{\bf A}({\bf U}_j^k)\sum_{l=0}^p{\bf U}_j^lD_{kl}+\delta_{kp}{\bf D}^-({\bf U}_j^p,{\bf U}_{j+1}^0)+\delta_{k0}{\bf D}^+({\bf U}_{j-1}^p,{\bf U}_{j}^0)=0,
\end{equation}

\noindent for all $j\in\mathbb{Z}$, $0\leq k\leq p$, and $t>0$. In section~\ref{sec:entropy_stable_flux}, we propose to modify the volume integral in (\ref{eq:semi-discr_DGSEM}) so as to satisfy an entropy balance. Note that the scheme (\ref{eq:semi-discr_ECPC_DGSEM}) satisfies a certain conservation property,

\begin{equation}\label{eq:mean_DGSEM}
 h\frac{d\langle{\bf u}\rangle_j}{dt}+\langle{\bf A}({\bf u}_h),d_x{\bf u}_h\rangle_j^p+{\bf D}^-({\bf U}_{j}^p,{\bf U}_{j+1}^0)+{\bf D}^+({\bf U}_{j-1}^p,{\bf U}_{j}^0)=0,
\end{equation}
\noindent for the cell averaged solution 

\begin{equation*}
 \langle{\bf u}\rangle_j(t) := \frac{1}{h}\int_{\kappa_j}{\bf u}_h(x,t)dx = \frac{1}{2}\sum_{k=0}^p\omega_k{\bf U}_j^k(t).
\end{equation*}

The numerical fluxes in fluctuation form satisfy the following consistency property

\begin{equation}\label{eq:consistent_flux}
 {\bf D}^\pm({\bf u},{\bf u}) = 0, \quad \forall {\bf u}\in\Omega^a,
\end{equation}
\noindent and may also satisfy the path-conservative property \cite{pares_06}

\begin{equation}\label{eq:castro_path_cons_flux}
 {\bf D}^-({\bf u}^-,{\bf u}^+)+{\bf D}^+({\bf u}^-,{\bf u}^+) = \int_0^1{\bf A}\big(\phib(s;{\bf u}^-,{\bf u}^+)\big)\partial_s\phib(s;{\bf u}^-,{\bf u}^+)ds,
\end{equation}
\noindent for a given path (\ref{eq:DLM_path}).

\subsection{Entropy stable numerical fluxes}\label{sec:entropy_stable_flux}

In the following, we use the usual terminology and denote by {\it entropy conservative} for the entropy-entropy flux pair $(\eta,q)$ in (\ref{eq:entropy_ineq_cont}), the numerical fluxes ${\bf D}_{ec}^\pm$ satisfying \cite{castro_etal_13}:

\begin{equation}\label{eq:entropy_conserv_flux}
 \etab'({\bf u}^-)^\top{\bf D}_{ec}^-({\bf u}^-,{\bf u}^+)+\etab'({\bf u}^+)^\top{\bf D}_{ec}^+({\bf u}^-,{\bf u}^+) = q({\bf u}^+)-q({\bf u}^-), \quad \forall {\bf u}^\pm\in\Omega^a.
\end{equation}

Furthermore, we will assume that the numerical fluxes at interfaces in (\ref{eq:semi-discr_DGSEM}) are {\it entropy stable} in the following sense:

\begin{equation}\label{eq:entropy_stable_flux}
 \etab'({\bf u}^-)^\top{\bf D}^-({\bf u}^-,{\bf u}^+)+\etab'({\bf u}^+)^\top{\bf D}^+({\bf u}^-,{\bf u}^+) \geq q({\bf u}^+)-q({\bf u}^-), \quad \forall {\bf u}^\pm\in\Omega^a.
\end{equation}

As done by Chen and Shu \cite{chen_shu_17} for hyperbolic conservation laws, we modify the volume integral in (\ref{eq:semi-discr_DGSEM}) to satisfy the entropy inequality at the semi-discrete level. The semi-discrete scheme now reads

\begin{equation}\label{eq:semi-discr_ECPC_DGSEM}
 \frac{\omega_kh}{2}\frac{d{\bf U}_j^k}{dt} + {\bf R}_j^k({\bf u}_h)=0, \quad \forall j\in\mathbb{Z}, \quad 0\leq k\leq p, \quad t>0,
\end{equation}
\noindent with

\begin{equation}\label{eq:discr_ECPC_DGSEM_BN_res}
 {\bf R}_j^k({\bf u}_h) = \omega_k\sum_{l=0}^p\tilde{\bf D}({\bf U}_j^k,{\bf U}_j^l)D_{kl}+\delta_{kp}{\bf D}^-({\bf U}_j^p,{\bf U}_{j+1}^0)+\delta_{k0}{\bf D}^+({\bf U}_{j-1}^p,{\bf U}_{j}^0),
\end{equation}
\noindent and

\begin{equation}\label{eq:ECPC_intvol_func}
 \tilde{\bf D}({\bf u}^-,{\bf u}^+):={\bf D}_{ec}^-({\bf u}^-,{\bf u}^+)-{\bf D}_{ec}^+({\bf u}^+,{\bf u}^-), \quad \forall {\bf u}^\pm\in\Omega^a,
\end{equation}
\noindent where ${\bf D}_{ec}^\pm(\cdot,\cdot)$ are some entropy conservative fluctuation fluxes (\ref{eq:entropy_conserv_flux}).

%
%
\section{Properties of the semi-discrete scheme}\label{sec:scheme_properties}

\subsection{Entropy stable scheme}

Theorem~\ref{th:entropy_inequ_DGSEM} proves a semi-discrete entropy inequality for the scheme (\ref{eq:semi-discr_ECPC_DGSEM}) together with entropy stable fluxes at interfaces, while Theorem~\ref{th:roe_type_fluxes} establishes high-order accuracy and the preservation of equation (\ref{eq:mean_DGSEM}) for the cell averaged solution.

\begin{theorem}[entropy stable DGSEM]\label{th:entropy_inequ_DGSEM}
Let $\tilde{\bf D}(\cdot,\cdot)$ defined in (\ref{eq:ECPC_intvol_func}) with ${\bf D}_{ec}^\pm(\cdot,\cdot)$ consistent (\ref{eq:consistent_flux}) and entropy conservative (\ref{eq:entropy_conserv_flux}) fluctuation fluxes, and let ${\bf D}^\pm(\cdot,\cdot)$ be consistent (\ref{eq:consistent_flux}) and entropy stable (\ref{eq:entropy_stable_flux}) fluctuation fluxes. Then, the semi-discrete DGSEM (\ref{eq:semi-discr_ECPC_DGSEM}) satisfies the following entropy inequality for the  pair $(\eta,q)$ in (\ref{eq:entropy_ineq_cont})

\begin{equation}\label{eq:entropy_inequ_DGSEM}
 h\frac{d\langle\eta\rangle_j}{dt}+Q({\bf U}_{j}^p,{\bf U}_{j+1}^0)-Q({\bf U}_{j-1}^p,{\bf U}_{j}^0)\leq0,
\end{equation}
\noindent with $\langle\eta\rangle_j=\sum_{k=0}^p\tfrac{\omega_k}{2}\eta({\bf U}_j^k)$ and either

\begin{equation}\label{eq:entropy_flux1}
 Q({\bf U}_{j}^p,{\bf U}_{j+1}^0) = q({\bf U}_{j}^p)+\etab'({\bf U}_{j}^p)^\top{\bf D}^-({\bf U}_j^p,{\bf U}_{j+1}^0),
\end{equation}
\noindent or

\begin{equation}\label{eq:entropy_flux2}
 Q({\bf U}_{j}^p,{\bf U}_{j+1}^0) = q({\bf U}_{j+1}^0)-\etab'({\bf U}_{j+1}^0)^\top{\bf D}^+({\bf U}_j^p,{\bf U}_{j+1}^0).
\end{equation}
\end{theorem}

\begin{proof}
 Left multiplying (\ref{eq:semi-discr_ECPC_DGSEM}) with $\etab'({\bf U}_j^k)$ and adding up over $0\leq k\leq p$, we obtain

\begin{equation*}
 h\frac{d\langle\eta\rangle_j}{dt}+\sum_{k,l}\omega_k\etab'({\bf U}_{j}^k)^\top\tilde{\bf D}({\bf U}_j^k,{\bf U}_j^l)D_{kl}+\etab'({\bf U}_j^p)^\top{\bf D}^-({\bf U}_j^p,{\bf U}_{j+1}^0)+\etab'({\bf U}_j^0)^\top{\bf D}^+({\bf U}_{j-1}^p,{\bf U}_{j}^0)=0,
\end{equation*}
\noindent where the second term may be transformed into

\begin{subeqnarray*}
 \sum_{k,l}\omega_k\etab'({\bf U}_{j}^k)^\top\tilde{\bf D}({\bf U}_j^k,{\bf U}_j^l)D_{kl} &\overset{(\ref{eq:ECPC_intvol_func})}{=}& \sum_{k,l}\omega_k\etab'({\bf U}_{j}^k)^\top\big({\bf D}_{ec}^-({\bf U}_j^k,{\bf U}_j^l) - {\bf D}_{ec}^+({\bf U}_j^l,{\bf U}_j^k)\big)D_{kl}\\
 &\overset{(\ref{eq:SBP})}{=}& \sum_{k,l}\omega_k\etab'({\bf U}_{j}^k)^\top{\bf D}_{ec}^-({\bf U}_j^k,{\bf U}_j^l)D_{kl} + \omega_l\etab'({\bf U}_{j}^k)^\top{\bf D}_{ec}^+({\bf U}_j^l,{\bf U}_j^k)D_{lk} \\
 &-& \delta_{kl}(\delta_{kp}-\delta_{k0})\etab'({\bf U}_{j}^k)^\top{\bf D}_{ec}^+({\bf U}_j^l,{\bf U}_j^k)\\
 &\overset{(\ref{eq:consistent_flux})}{\underset{k\leftrightarrow l}{=}}& \sum_{k,l}\omega_k\big(\etab'({\bf U}_{j}^k)^\top{\bf D}_{ec}^-({\bf U}_j^k,{\bf U}_j^l) + \etab'({\bf U}_{j}^l)^\top{\bf D}_{ec}^+({\bf U}_j^k,{\bf U}_j^l)\big)D_{kl}\\
 &\overset{(\ref{eq:entropy_conserv_flux})}{=}& \sum_{k,l}\omega_k\big(q({\bf U}_{j}^l)-q({\bf U}_j^k) \big)D_{kl}\\
 &\overset{(\ref{eq:interp_lag_unite_deriv})}{=}& \sum_{k,l}\omega_k q({\bf U}_{j}^l)D_{kl}\\
 &=& q({\bf U}_{j}^p)-q({\bf U}_{j}^0),
\end{subeqnarray*}

\noindent where $k\leftrightarrow l$ indicates an inversion of indices $k$ and $l$ in some of the terms. We thus obtain

\begin{equation*}
 h\frac{d\langle\eta\rangle_j}{dt}+q({\bf U}_{j}^p)-q({\bf U}_{j}^0)+\etab'({\bf U}_j^p)^\top{\bf D}^-({\bf U}_j^p,{\bf U}_{j+1}^0)+\etab'({\bf U}_j^0)^\top{\bf D}^+({\bf U}_{j-1}^p,{\bf U}_{j}^0)=0,
\end{equation*}
\noindent and using (\ref{eq:entropy_flux1}) we deduce

\begin{subeqnarray*}
 h\frac{d\langle\eta\rangle_j}{dt}+Q({\bf U}_{j}^p,{\bf U}_{j+1}^0)-Q({\bf U}_{j-1}^p,{\bf U}_{j}^0) &=& q({\bf U}_{j}^0)-\etab'({\bf U}_{j}^0)^\top{\bf D}^+({\bf U}_{j-1}^p,{\bf U}_{j}^0) \\
 &-& q({\bf U}_{j-1}^p)-\etab'({\bf U}_{j-1}^p)^\top{\bf D}^-({\bf U}_{j-1}^p,{\bf U}_{j}^0) \overset{(\ref{eq:entropy_stable_flux})}{\leq} 0.
\end{subeqnarray*}

The same holds with (\ref{eq:entropy_flux2}).\quad$\square$
\end{proof}

Entropy conservation then results as an immediate consequence.

\begin{corollary}[entropy conservative fluxes]
 Under the assumptions of Theorem~\ref{th:entropy_inequ_DGSEM}, the semi-discrete DGSEM (\ref{eq:semi-discr_ECPC_DGSEM}) is entropy conservative iff. the numerical fluxes at interfaces are entropy conservative (\ref{eq:entropy_conserv_flux}). The numerical entropy flux reads

\begin{subeqnarray*}
 Q({\bf U}_{j}^p,{\bf U}_{j+1}^0) &=& q({\bf U}_{j}^p)+\etab'({\bf U}_{j}^p)^\top{\bf D}_{ec}^-({\bf U}_j^p,{\bf U}_{j+1}^0), \\
 &=& q({\bf U}_{j+1}^0)-\etab'({\bf U}_{j+1}^0)^\top{\bf D}_{ec}^+({\bf U}_j^p,{\bf U}_{j+1}^0).
\end{subeqnarray*}
\end{corollary}

High-order accuracy and the conservation-like property (\ref{eq:mean_DGSEM}) require further assumptions on the form of the entropy conservative fluxes (\ref{eq:ECPC_intvol_func}) which are summarized in Theorem~\ref{th:roe_type_fluxes} below. We stress that this form of fluctuation fluxes is fairly general and includes for instance skew-symmetric splittings (see Corollary~\ref{th:skew-sym-splitting}).

\begin{theorem}\label{th:roe_type_fluxes}
 Under the assumptions of Theorem~\ref{th:entropy_inequ_DGSEM} and further assuming that the entropy conservative fluctuation fluxes have the following form

\begin{subeqnarray}\label{eq:roe_type_fluxes}
 {\bf D}_{ec}^\pm({\bf u}^-,{\bf u}^+) &=& {\bf\cal A}^\pm({\bf u}^-,{\bf u}^+)\du {\bf u} \df,\\
 {\bf\cal A}({\bf u}^-,{\bf u}^+) &:=& {\bf\cal A}^-({\bf u}^-,{\bf u}^+)+{\bf\cal A}^+({\bf u}^-,{\bf u}^+),\\
 {\bf\cal A}({\bf u}^-,{\bf u}^+) + {\bf\cal A}({\bf u}^+,{\bf u}^-)  &=& {\bf A}({\bf u}^-) + {\bf A}({\bf u}^+),\\
 {\bf\cal A}({\bf u},{\bf u}) &=& {\bf A}({\bf u}),
\end{subeqnarray}
\noindent for all ${\bf u}^\pm$ and ${\bf u}$ in $\Omega^a$. Then, the semi-discrete DGSEM (\ref{eq:semi-discr_ECPC_DGSEM}) is a high-order approximation in space of the nonconservative equation (\ref{eq:NL_model_pb}a) which satisfies (\ref{eq:mean_DGSEM}).
\end{theorem}

\begin{proof}
First, to prove accuracy, it is sufficient to prove that the volume integral in (\ref{eq:semi-discr_ECPC_DGSEM}) is a high-order approximation of ${\bf A}({\bf u})\partial_x{\bf u}$ at points $x_j^k$, $0\leq k\leq p$, for smooth enough solutions ${\bf u}$. Let $\pi_h^p:L^2(\Omega_h)\ni u\mapsto\pi_h^p(u)\in{\cal V}_h^p$ be the Lagrange projection onto ${\cal V}_h^p$ associated to nodes (\ref{eq:lagrange_poly}). Since the Lagrange interpolation error is of order ${\cal O}(h^{p+1})$, we have for $u$ and $v$ in ${\cal C}^{p+1}(\Omega_h)$:

\begin{equation}\label{eq:lag_interp_deriv_product}
 d_x\pi_h^p(uv)(x)=u(x)d_xv(x)+v(x)d_xu(x) + {\cal O}(h^p), \quad \forall x\in\Omega_h.
\end{equation}

Let $t>0$, introducing the interpolation polynomial ${\bf a}_h^k(x):=\sum_{l=0}^p{\bf\cal A}^-({\bf U}_j^k,{\bf U}_j^l)\phi_j^l(x)$, we have ${\bf a}_h^k(x_j^k)={\bf\cal A}^-({\bf U}_j^k,{\bf U}_j^k)$ and $d_x{\bf a}_h^k(x_j^k)=\sum_{l=0}^p{\bf\cal A}^-({\bf U}_j^k,{\bf U}_j^l)d_x\phi_j^l(x_j^k)$. Using (\ref{eq:lag_interp_deriv_product}) for the product ${\bf a}_h^k{\bf u}_h$, we obtain

\begin{subeqnarray*}
 \frac{2}{h}\sum_{l=0}^p{\bf\cal A}^-({\bf U}_j^k,{\bf U}_j^l){\bf U}_j^lD_{kl} = {\bf\cal A}^-({\bf U}_j^k,{\bf U}_j^k)d_x{\bf u}_h(x_j^k) + d_x{\bf a}_h^k(x_j^k){\bf U}_j^k + {\cal O}(h^p).
\end{subeqnarray*}

Applying the same rule for ${\bf a}_h^k(x):=\sum_{l=0}^p{\bf\cal A}^+({\bf U}_j^l,{\bf U}_j^k)\phi_j^l(x)$, we finally obtain

\begin{subeqnarray}\label{eq:approx_chain_rule}
 \frac{2}{h}\sum_{l=0}^p{\bf\cal A}^-({\bf U}_j^k,{\bf U}_j^l)({\bf U}_j^l-{\bf U}_j^k)D_{kl} &=& {\bf\cal A}^-({\bf U}_j^k,{\bf U}_j^k)d_x{\bf u}_h(x_j^k) + {\cal O}(h^p),\\
 \frac{2}{h}\sum_{l=0}^p{\bf\cal A}^+({\bf U}_j^l,{\bf U}_j^k)({\bf U}_j^l-{\bf U}_j^k)D_{kl} &=& {\bf\cal A}^+({\bf U}_j^k,{\bf U}_j^k)d_x{\bf u}_h(x_j^k) + {\cal O}(h^p).
\end{subeqnarray}

We thus have

\begin{subeqnarray*}
 \frac{2}{h}\sum_{l=0}^p\tilde{\bf D}({\bf U}_j^k,{\bf U}_j^l)D_{kl} &\overset{(\ref{eq:roe_type_fluxes}a)}{=}& \frac{2}{h}\sum_{l=0}^p\big({\bf\cal A}^-({\bf U}_j^k,{\bf U}_j^l)+{\bf\cal A}^+({\bf U}_j^l,{\bf U}_j^k)\big)({\bf U}_j^l-{\bf U}_j^k)D_{kl} \\
 &\overset{(\ref{eq:approx_chain_rule})}{=}& \big({\bf\cal A}^-({\bf U}_j^k,{\bf U}_j^k)+{\bf\cal A}^+({\bf U}_j^k,{\bf U}_j^k)\big)d_x{\bf u}_h(x_j^k) + {\cal O}(h^{p})\\
 &\overset{(\ref{eq:roe_type_fluxes}b)}{=}& {\bf\cal A}({\bf U}_j^k,{\bf U}_j^k)d_x{\bf u}_h(x_j^k) + {\cal O}(h^{p})\\
 &\overset{(\ref{eq:roe_type_fluxes}d)}{=}& {\bf A}({\bf U}_j^k)d_x{\bf u}_h(x_j^k) + {\cal O}(h^{p}).
\end{subeqnarray*}

Then, to obtain (\ref{eq:mean_DGSEM}), we add up (\ref{eq:semi-discr_ECPC_DGSEM}) over $0\leq k\leq p$ and obtain

\begin{equation*}
  h\frac{d\langle{\bf u}\rangle_j}{dt}+\sum_{k=0}^p\omega_k\sum_{l=0}^p\tilde{\bf D}({\bf U}_j^k,{\bf U}_j^l)D_{kl}+{\bf D}^-({\bf U}_{j}^p,{\bf U}_{j+1}^0)+{\bf D}^+({\bf U}_{j-1}^p,{\bf U}_{j}^0)=0,
\end{equation*}
\noindent where the second term may be transformed into

\begin{subeqnarray*}
 \sum_{k,l}\omega_k\tilde{\bf D}({\bf U}_j^k,{\bf U}_j^l)D_{kl} &\overset{(\ref{eq:ECPC_intvol_func})}{=}& \sum_{k,l}\omega_k\big({\bf D}_{ec}^-({\bf U}_j^k,{\bf U}_j^l)-{\bf D}_{ec}^+({\bf U}_j^l,{\bf U}_j^k)\big)D_{kl}\\
 &\overset{(\ref{eq:SBP})}{=}& \sum_{k,l}\omega_k{\bf D}_{ec}^-({\bf U}_j^k,{\bf U}_j^l)D_{kl}+\omega_l{\bf D}_{ec}^+({\bf U}_j^l,{\bf U}_j^k)D_{lk}\\
 &-& \delta_{kl}(\delta_{kp}-\delta_{k0}){\bf D}_{ec}^+({\bf U}_j^l,{\bf U}_j^k)\\
 &\overset{(\ref{eq:consistent_flux})}{\underset{k\leftrightarrow l}{=}}& \sum_{k,l}\omega_k\big({\bf D}_{ec}^-({\bf U}_j^k,{\bf U}_j^l)+{\bf D}_{ec}^+({\bf U}_j^k,{\bf U}_j^l)\big)D_{kl}\\
 &\overset{(\ref{eq:roe_type_fluxes}a,b)}{=}& \sum_{k,l}\omega_k{\bf\cal A}({\bf U}_j^k,{\bf U}_j^l)({\bf U}_j^l-{\bf U}_j^k)D_{kl}\\
 &\overset{(\ref{eq:SBP})}{\underset{(\ref{eq:roe_type_fluxes}d)}{=}}& \sum_{k,l}\omega_k{\bf\cal A}({\bf U}_j^k,{\bf U}_j^l){\bf U}_j^lD_{kl}+\omega_l{\bf\cal A}({\bf U}_j^k,{\bf U}_j^l){\bf U}_j^kD_{lk}\\
 &-& {\bf A}({\bf U}_j^p){\bf U}_j^p + {\bf A}({\bf U}_j^0){\bf U}_j^0 \\
 &\overset{(\ref{eq:roe_type_fluxes}c)}{\underset{k\leftrightarrow l}{=}}& \sum_{k,l}\omega_k\big({\bf A}({\bf U}_j^k)+{\bf A}({\bf U}_j^l)\big){\bf U}_j^lD_{kl}  - {\bf A}({\bf U}_j^p){\bf U}_j^p + {\bf A}({\bf U}_j^0){\bf U}_j^0 \\
 &\overset{(\ref{eq:SBP})}{=}& \sum_{k,l}\omega_k{\bf A}({\bf U}_j^k){\bf U}_j^lD_{kl}, \\
 &=& \langle{\bf A}({\bf u}_h),d_x{\bf u}_h\rangle_j^p,
\end{subeqnarray*}
\noindent which completes the proof.\quad$\square$
\end{proof}

Now, we consider sequential splittings of the nonconservative product for smooth solutions of the form

\begin{equation}\label{eq:splitting_Au_x}
 {\bf A}\partial_x{\bf u} = \alpha{\bf A}\partial_x{\bf u} + (1-\alpha)\big(\partial_x({\bf A}{\bf u})-(\partial_x{\bf A}){\bf u}\big), \quad 0\leq\alpha\leq1.
\end{equation}

Entropy stable schemes based on the above decomposition fall into the assumptions of Theorem~\ref{th:roe_type_fluxes} as stated below.

\begin{corollary}[skew-symmetric splitting]\label{th:skew-sym-splitting}
Enropy conservative fluxes in (\ref{eq:entropy_conserv_flux}) for the splitting (\ref{eq:splitting_Au_x}) read

\begin{equation}\label{eq:skew-sym-splitting}
 {\bf D}_{ec}^\pm({\bf u}^-,{\bf u}^+)={\bf\cal A}^\pm({\bf u}^-,{\bf u}^+)\du {\bf u} \df, \quad {\bf\cal A}^\pm({\bf u}^-,{\bf u}^+) = \frac{1}{2}\big(\alpha{\bf A}({\bf u}^\pm)+(1-\alpha){\bf A}({\bf u}^\mp)\big),
\end{equation}
\noindent and constitute particular cases of the high-order entropy conservative fluxes (\ref{eq:roe_type_fluxes}) of Theorem~\ref{th:roe_type_fluxes}.
\end{corollary}

\begin{proof}
First, using (\ref{eq:skew-sym-splitting}) and (\ref{eq:ECPC_intvol_func}) to evaluate the volume integral in (\ref{eq:semi-discr_ECPC_DGSEM}), we obtain

\begin{subeqnarray*}
 \tilde{\bf D}({\bf u}^-,{\bf u}^+) &=& \big(\alpha{\bf A}({\bf u}^-)+(1-\alpha){\bf A}({\bf u}^+)\big)\du {\bf u} \df \\
 &=& \alpha{\bf A}({\bf u}^-)({\bf u}^+-{\bf u}^-) + (1-\alpha){\bf A}({\bf u}^+)({\bf u}^+-{\bf u}^-) \\
 &=& \alpha{\bf A}({\bf u}^-){\bf u}^+ + (1-\alpha)\big({\bf A}({\bf u}^+){\bf u}^+-{\bf A}({\bf u}^+){\bf u}^-\big),
\end{subeqnarray*}
\noindent since from (\ref{eq:interp_lag_unite_deriv}) the term ${\bf A}({\bf u}^-){\bf u}^-$ has no contribution to the volume integral. The above relation implies that (\ref{eq:ECPC_intvol_func}) is a volume discretization of the RHS of (\ref{eq:splitting_Au_x}).

Then, from (\ref{eq:roe_type_fluxes}b) we have

\begin{equation*}
 {\bf\cal A}({\bf u}^-,{\bf u}^+) = {\bf\cal A}^-({\bf u}^-,{\bf u}^+) + {\bf\cal A}^+({\bf u}^-,{\bf u}^+) = \tfrac{1}{2}\big({\bf A}({\bf u}^-)+{\bf A}({\bf u}^+)\big),
\end{equation*}
\noindent which indeed satisfies (\ref{eq:roe_type_fluxes}c,d).\quad$\square$
\end{proof}

\subsection{Entropy conservative fluxes for conservation laws}\label{sec:cons_flux}

In the particular case where (\ref{eq:NL_model_pb}) reduces to a conservation law, \ie, ${\bf A}({\bf u})={\bf f}'({\bf u})$, it has been shown in \cite{chen_shu_17} that it is possible to satisfy the entropy inequality (\ref{eq:entropy_inequ_DGSEM}) by using the entropy conservative fluxes ${\bf h}_{ec}({\bf u}^-, {\bf u}^+)$ from Tadmor \cite{tadmor87} which satisfy

\begin{subeqnarray*}
 \du \etab'\df^\top {\bf h}_{ec}({\bf u}^-,{\bf u}^+) &=& \du\etab'^\top{\bf f}-q \df, \quad \forall {\bf u}^\pm\in\Omega^a, \\
 {\bf h}_{ec}({\bf u},{\bf u}) &=& {\bf f}({\bf u}), \quad \forall {\bf u}\in\Omega^a.
\end{subeqnarray*}

The link between fluctuation fluxes and conservative fluxes reads

\begin{subeqnarray*}
 {\bf h}_{ec}({\bf u}^-,{\bf u}^+) =  {\bf f}({\bf u}^-) + {\bf D}_{ec}^-({\bf u}^-,{\bf u}^+) = {\bf f}({\bf u}^+) - {\bf D}_{ec}^+({\bf u}^-,{\bf u}^+),
\end{subeqnarray*}
\noindent from which we deduce that

\begin{equation}\label{eq:ECPC_link_noncons_cons}
 {\bf D}_{ec}^-({\bf u}^-,{\bf u}^+)-{\bf D}_{ec}^+({\bf u}^+,{\bf u}^-) = {\bf h}_{ec}({\bf u}^-,{\bf u}^+) + {\bf h}_{ec}({\bf u}^+,{\bf u}^-) - 2{\bf f}({\bf u}^-),
\end{equation}
\noindent and using (\ref{eq:interp_lag_unite_deriv}) the volume integral in (\ref{eq:semi-discr_ECPC_DGSEM}) becomes

\begin{equation}\label{eq:ECPC_intvol_func_cons}
 \omega_k\sum_{l=0}^p\tilde{\bf D}({\bf U}_j^k,{\bf U}_j^l)D_{kl} = \omega_k\sum_{l=0}^p\big({\bf h}_{ec}({\bf U}_j^k,{\bf U}_j^l) + {\bf h}_{ec}({\bf U}_j^l,{\bf U}_j^k)\big)D_{kl}.
\end{equation}

In \cite{chen_shu_17}, a slightly different choice has been made: $\tilde{\bf D}({\bf u}^-,{\bf u}^+) := 2{\bf h}_{ec}({\bf u}^-,{\bf u}^+)$, where ${\bf h}_{ec}(\cdot,\cdot)$ is assumed to be symmetric. In fact, it may be easily verified that the properties of Theorem 3.3 in \cite{chen_shu_17} also hold with (\ref{eq:ECPC_intvol_func_cons}) which may be seen as a generalization of the framework of entropy stable DGSEM to nonsymmetric entropy conservative fluxes by using the symmetrizer $({\bf h}_{ec}({\bf u}^-,{\bf u}^+) + {\bf h}_{ec}({\bf u}^+,{\bf u}^-))/2$.

%
%
\section{Examples}\label{sec:examples}

In this section we consider different nonconservative scalar equations and systems in one space dimension and provide each time examples of entropy conservative numerical fluxes that fall into the category considered in Theorems~\ref{th:entropy_inequ_DGSEM} and \ref{th:roe_type_fluxes} in section~\ref{sec:entropy_stable_flux}. We give a more detailed description of examples \ref{sec:noncons_2_pb} and \ref{sec:BN_isos} that will be used in the numerical experiments of section \ref{sec:num_xp}. In the following, it is convenient to introduce the average operator $\ol{u}:=\tfrac{u^-+u^+}{2}$.

\subsection{Burgers equation}

The Burgers equation in nonconservative form reads

\begin{equation*}
 \partial_tu+u\partial_xu=0,
\end{equation*}
\noindent with entropy $\eta(u)=\tfrac{u^2}{2}$ and entropy flux $q(u)=\tfrac{u^3}{3}$. Entropy conservative fluctuation fluxes of the form (\ref{eq:roe_type_fluxes}) read

\begin{equation*}
 D_{ec}^-(u^-,u^+) = \frac{2u^-+u^+}{6}\du u\df, \quad D_{ec}^+(u^-,u^+) = \frac{u^-+2u^+}{6}\du u\df.
\end{equation*}

Using (\ref{eq:ECPC_link_noncons_cons}), with $f(u)=\tfrac{u^2}{2}$, and looking for an equivalent symmetric entropy conservative flux for conservative equations, we obtain

\begin{equation*}
 h_{ec}(u^-,u^+) = \frac{D_{ec}^-(u^-,u^+) - D_{ec}^+(u^-,u^+) + 2f(u^-)}{2} = \frac{(u^-)^2+u^⁻u^++(u^+)^2}{6},
\end{equation*}
\noindent which corresponds to the entropy conservative skew-symmetric splitting of the Burgers equation \cite{tadmor84}.

\subsection{coupled Burgers equation}
The following nonconservative system was first proposed in \cite{berthon_02}:

\begin{subeqnarray*}
 \partial_tu+u\partial_x(u+v) &=& 0,\\
 \partial_tv+v\partial_x(u+v) &=& 0,
\end{subeqnarray*}
\noindent where we recover the Burgers equation for the sum $u+v$. Entropy and entropy flux are therefore $\eta({\bf u})=\tfrac{(u+v)^2}{2}$ and $q({\bf u})=\tfrac{(u+v)^3}{3}$. Entropy conservative fluctuation fluxes of the form (\ref{eq:roe_type_fluxes}) may also be derived:

\begin{equation*}
 {\bf D}_{ec}^-({\bf u}^-,{\bf u}^+) = \frac{\du u+v\df}{6}\begin{pmatrix}2u^-+u^+\\ 2v^-+v^+\end{pmatrix},
 \quad {\bf D}_{ec}^+({\bf u}^-,{\bf u}^+) = \frac{\du u+v\df}{6}\begin{pmatrix}u^-+2u^+\\ v^-+2v^+\end{pmatrix},
\end{equation*}
\noindent which correspond to the path-conservative and entropy conservative fluxes derived in \cite{castro_etal_13}.

\subsection{Nonconservative product associated to a LD field}\label{sec:noncons_2_pb}
Let us introduce the following nonlinear hyperbolic system representative of two-phase flow problems where the LD characteristic field plays the role of interface velocity \cite{coquel_etal_02}:

\begin{subeqnarray}\label{eq:Sys22_LD_noncons}
 \partial_tu+g({\bf u})\partial_xu &=& 0,\\
 \partial_tv+\partial_xf({\bf u}) &=& 0,
\end{subeqnarray}
\noindent with $g({\bf u})=u+v$ and $f({\bf u})=\tfrac{v^2-u^2}{2}$. The eigenvalues are $g({\bf u})$ associated to the LD field and $v$ associated to a genuinely nonlinear field so the system is strictly hyperbolic over the set of states $\Omega^a=\{(u,v)^\top\in\mathbb{R}^2:u>0\}$. It satisfies an entropy inequality for the pair $\eta({\bf u})=\tfrac{(u+v)^2}{2}$ and $q({\bf u})=\tfrac{(u+v)^3}{3}$.

Entropy conservative fluctuation fluxes are

\begin{subeqnarray}\label{eq:2_2_sys_EC_flux}
 {\bf D}_{ec}^-({\bf u}^-,{\bf u}^+) &=& \frac{1}{6}\begin{pmatrix}\big(2g({\bf u}^-)+g({\bf u}^+)\big)\du u\df\\ (2v^-+v^+)\du v\df-(2u^-+u^+)\du u\df\end{pmatrix},\\
 {\bf D}_{ec}^-({\bf u}^-,{\bf u}^+) &=& \frac{1}{6}\begin{pmatrix}\big(g({\bf u}^-)+2g({\bf u}^+)\big)\du u\df\\ (v^-+2v^+)\du v\df-(u^-+2u^+)\du u\df\end{pmatrix}.
\end{subeqnarray}

Note that the regularized system

\begin{equation*}
 \partial_tu+g({\bf u})\partial_xu = \epsilon\partial^2_{xx}u,\quad 
 \partial_tv+\partial_xf({\bf u}) =  \epsilon\partial^2_{xx}v,
\end{equation*}
\noindent with $\epsilon>0$ gives

\begin{equation*}
 \partial_t\eta({\bf u}) + \partial_xq({\bf u}) - \epsilon\partial^2_{xx}\eta({\bf u}) = - \epsilon\big((\partial_xu)^2+(\partial_xv)^2\big) \leq 0,
\end{equation*}
\noindent so the associated viscous profiles will give the physically admissible solutions in the limit $\epsilon=0^+$. Using this result for numerical purposes, we design the following entropy stable flux

\begin{equation}\label{eq:2_2_sys_ES_flux}
{\bf D}^\pm({\bf u}^-,{\bf u}^+) = \begin{pmatrix}\tfrac{2g({\bf u}^\pm)+g({\bf u}^\mp)}{6}\du u\df\\ \pm\big(f({\bf u}^\pm)-\hat{h}({\bf u}^-,{\bf u}^+)\big) \end{pmatrix} \pm \epsilon_v\du{\bf u}\df, \quad \hat{h}({\bf u}^-,{\bf u}^+) = \tfrac{f({\bf u}^-)+f({\bf u}^+)}{2} - \tfrac{\beta_s}{2}\du v\df,
\end{equation}

\noindent with numerical parameters $\epsilon_v\geq0$ and $\beta_s\geq0$. Setting $\epsilon_v=0$, it may be checked that the fluctuations fluxes in (\ref{eq:2_2_sys_ES_flux}) are entropy conservative providing that $\beta_s=(\du v\df-\du u\df^2/\du v\df)/6$. In practice, we set $\beta_s=\max\big(|v^\pm|,|g({\bf u}^\pm)|,(\du v\df-\du u\df^2/\du v\df)/6,0\big)$ and $\epsilon_v>0$ to get an entropy stable flux.
\subsection{Euler equations in Lagrangian coordinates}
The Euler equations in Lagrangian coordinates may be written in nonconservative form:

\begin{subeqnarray*}
 \partial_t\tau-\partial_xu &=& 0,\\
 \partial_tu+\partial_x\mathrm{p} &=& 0,\\
 \partial_te+\mathrm{p}\partial_xu &=& 0,
\end{subeqnarray*}
\noindent with $\tau$ the specific volume, $u$ the velocity, $e$ the specific internal energy. The equations are supplemented with a general equation of states for the pressure $\mathrm{p}=\mathrm{p}(\tau,e)$ and admissible solutions satisfy the entropy inequality

\begin{equation*}
 \partial_t\mathrm{s} \geq 0,
\end{equation*}
\noindent with $\mathrm{T}d\mathrm{s} = de + \mathrm{p}d\tau$, and $\mathrm{T}$ the temperature.

Entropy conservative fluctuation fluxes are

\begin{equation*}
 {\bf D}_{ec}^-({\bf u}^-,{\bf u}^+) = \frac{1}{2}\begin{pmatrix}-\du u\df\\ \du \mathrm{p}\df\\ \mathrm{p}^-\du u\df\end{pmatrix},\quad {\bf D}_{ec}^+({\bf u}^-,{\bf u}^+) = \frac{1}{2}\begin{pmatrix}-\du u\df\\ \du \mathrm{p}\df\\ \mathrm{p}^+\du u\df\end{pmatrix}.
\end{equation*}

Note that these fluxes are different from the path-conservative Roe-type method with straight-line paths in $\tau$, $u$ and $\mathrm{p}$ \cite{abgrall_karni_10,chalons_coquel_17,toumi_92} where the fluctuation fluxes read

\begin{equation*}
 {\bf D}_{roe}^\pm({\bf u}^-,{\bf u}^+) = {\bf\cal A}({\bf u}^-,{\bf u}^+)^\pm\du{\bf u}\df = {\bf A}(\tilde{\bf v})\du{\bf u}\df,
\end{equation*}
\noindent with $\tilde{\bf v}=(\ol\tau,\ol u,\ol{\mathrm{p}})^\top$. 

%
\subsection{One-pressure model of spray dynamics}
We now consider the one-pressure two-velocity four equations system for modeling the dynamics of a spray of liquid droplets in a gas at thermodynamic equilibrium \cite{sainsaulieu_91,stewart_wendroff_84}. Let $\rho_g$ be the gas density, $\rho_l>0$ the constant and uniform liquid density, $\alpha$ the void fraction of the gas, and $u_g$ and $u_l$ the velocities of the gas and liquid phases. The variables obey the following hyperbolic system

\begin{subeqnarray*}
 \partial_t(\alpha\rho_g) + \partial_x(\alpha\rho_gu_g) &=& 0,\\
 \partial_t(\alpha\rho_gu_g) + \partial_x(\alpha_g\rho_gu_g^2)+\alpha\partial_x\mathrm{p} &=& 0,\\
 \partial_t\big((1-\alpha)\rho_l\big) + \partial_x\big((1-\alpha)\rho_lu_l\big) &=& 0,\\
 \partial_t\big((1-\alpha)\rho_lu_l\big) + \partial_x\big((1-\alpha)\rho_lu_l^2\big) + (1-\alpha)\partial_x\mathrm{p} + \partial_x\theta &=& 0,\\
\end{subeqnarray*}
\noindent over the set of states $\Omega^a=\{{\bf u}\in\mathbb{R}^4:\;\rho_g>0,0<\alpha<1\}$. The gas pressure $\mathrm{p}=\mathrm{p}(\rho_g)$ satisfies $\mathrm{p}'(\rho_g)>0$, and $\theta(\alpha)=\theta_0(1-\alpha)^\delta$, with $1<\delta<2$, where $\theta_0$ denotes the total pressure of the gas on a droplet. The system satisfies an entropy inequality (\ref{eq:entropy_ineq_cont}) for the pair

\begin{subeqnarray*}
 \eta({\bf u}) &=& \alpha\rho_g\Big(\frac{u_g^2}{2}+e(\rho_g)\Big) + (1-\alpha)\rho_l\frac{u_l^2}{2} + \frac{\theta(\alpha)}{\delta-1}, \\
 q({\bf u}) &=& \alpha\rho_g\Big(\frac{u_g^2}{2}+\mathrm{h}(\rho_g)\Big)u_g + (1-\alpha)\Big(\rho_l\frac{u_l^2}{2}+\mathrm{p}(\rho_g)\Big)u_l + \frac{\delta}{\delta-1}\theta(\alpha)u_l,
\end{subeqnarray*}
\noindent where $\rho_g^2e'(\rho_g)=\mathrm{p}(\rho_g)$ and $\mathrm{h}(\rho_g)=e(\rho_g)+\mathrm{p}(\rho_g)/\rho_g$. It can be checked that the following fluxes are entropy conservative:

\begin{subeqnarray*}
 {\bf D}_{ec}^-({\bf u}^-,{\bf u}^+) = \begin{pmatrix}
 \hat{h}_g-\alpha^-\rho_g^-u_g^- \\ \hat{h}_g\ol{u_g}+\ol{\alpha}\;\ol{\mathrm{p}} - \alpha^-(\rho_g^-(u_g^-)^2+\mathrm{p}^-) - \tfrac{\mathrm{p}^-}{2}\du\alpha\df \\ \hat{h}_l-(1-\alpha^-)\rho_lu_l^- \\ \hat{h}_l\ol{u_l}+\ol{1-\alpha}\ol{\mathrm{p}} - (1-\alpha^-)(\rho_l(u_l^-)^2+\mathrm{p}^-) + \tfrac{\mathrm{p}^-}{2}\du\alpha\df + \tfrac{\du\theta\df}{2}
\end{pmatrix}, \\
 {\bf D}_{ec}^+({\bf u}^-,{\bf u}^+) = \begin{pmatrix}
 \alpha^+\rho_g^+u_g^+-\hat{h}_g\\ \alpha^+(\rho_g^+(u_g^+)^2+\mathrm{p}^+)-\hat{h}_g\ol{u_g}-\ol{\alpha}\;\ol{\mathrm{p}} - \tfrac{\mathrm{p}^+}{2}\du\alpha\df \\ (1-\alpha^+)\rho_lu_l^+-\hat{h}_l \\ (1-\alpha^+)(\rho_l(u_l^+)^2+\mathrm{p}^+)-\hat{h}_l\ol{u_l}-\ol{1-\alpha}\ol{\mathrm{p}} + \tfrac{\mathrm{p}^+}{2}\du\alpha\df + \tfrac{\du\theta\df}{2}
\end{pmatrix},
\end{subeqnarray*}
\noindent where

\begin{subeqnarray*}
 \hat{h}_g &=& \left\{\begin{array}{ll}
 \ol{\alpha}\;\ol{u_g}\frac{\du\mathrm{p}(\rho_g)\df}{\du\mathrm{h}(\rho_g)\df} & \mbox{if } \rho_g^- \neq \rho_g^+, \\ \ol{\alpha}\;\ol{u_g}\rho_g & \mbox{if } \rho_g^- = \rho_g^+ = \rho_g,
\end{array}\right. \\
 \hat{h}_l &=& \left\{\begin{array}{ll}
 \rho_l\frac{\ol{1-\alpha}\;\ol{u_l}\du\mathrm{p}(\rho_g)\df+\ol{u_l}\du\theta(\alpha)\df}{\du\mathrm{p}(\rho_g)\df+\tfrac{\delta}{\delta-1}\du\theta(\alpha)\df} & \mbox{if } \rho_g^- \neq \rho_g^+ \mbox{ or } \alpha^- \neq \alpha^+, \\ \rho_l(1-\alpha)\ol{u_l} & \mbox{if } \rho_g^- = \rho_g^+ = \rho_g \mbox{ and } \alpha^- = \alpha^+ = \alpha.
\end{array}\right. 
\end{subeqnarray*}

\subsection{Isentropic Baer-Nunziato model}\label{sec:BN_isos}
We finally consider the two-pressure two-velocity isentropic model \cite{baer_nunziato86,ambroso_etal_CMS_08} with void fractions $\alpha_i$, densities $\rho_i$, velocities $u_i$, and general equations of states $\mathrm{p}_i=\mathrm{p}_i(\rho_i)$ with $\mathrm{p}_i'(\rho_i)>0$ and $\mathrm{p}_i''(\rho_i)<0$ for phases $i=1,2$. It is useful to introduce the specific internal energy $e_i$ and enthalpy $\mathrm{h}_i$ of both phases defined by $\rho_i^2e_i'(\rho_i)=\mathrm{p}_i(\rho_i)$ and $\rho_i\mathrm{h}_i(\rho_i)=\rho_ie_i(\rho_i)+\mathrm{p}_i(\rho_i)$. Likewise, we introduce the speeds of sound $c_i^2(\rho_i)=\mathrm{p}_i'(\rho_i)$.

\subsubsection{Two-phase flow model}\label{sec:BN_model}
Neglecting source terms modeling relaxation mechanisms, the governing equations have the form (\ref{eq:PDE_cons_noncons}) with

\begin{equation}\label{eq:BN}
 {\bf u} = \begin{pmatrix}\alpha_1 \\ \alpha_1\rho_1 \\ \alpha_1\rho_1u_1 \\ \alpha_2\rho_2 \\ \alpha_2\rho_2u_2 \end{pmatrix}, \quad
{\bf f}({\bf u}) = \begin{pmatrix} 0 \\ \alpha_1\rho_1u_1 \\ \alpha_1(\rho_1u_1^2+\mathrm{p}_1) \\ \alpha_2\rho_2u_2 \\ \alpha_2(\rho_2u_2^2+\mathrm{p}_2)\end{pmatrix}, \quad
{\bf c}({\bf u})\partial_x{\bf u} = \begin{pmatrix} u_2 \\ 0 \\ -\mathrm{p}_1 \\ 0 \\ \mathrm{p}_1 \end{pmatrix}\partial_x\alpha_1,
\end{equation}
\noindent where $u_2$ and $\mathrm{p}_1$ have been chosen as closure laws for the interface velocity and pressure, respectively. Both phases are assumed to satisfy the saturation condition

\begin{equation}\label{eq:saturation_condition}
 \alpha_1+\alpha_2=1.
\end{equation}

The set of states is $\Omega^a=\{{\bf u}\in\mathbb{R}^5:\;\rho_i>0, \alpha_i>0, i=1,2\}$ and the  system satisfies an entropy inequality (\ref{eq:entropy_ineq_cont}) for the pair

\begin{equation}\label{eq:BN_entropy_pair}
 \eta({\bf u}) = \sum_{i=1}^2 \alpha_i\rho_i\Big(\frac{u_i^2}{2}+e_i(\rho_i)\Big), \quad q({\bf u}) = \sum_{i=1}^2 \alpha_i\rho_i\Big(\frac{u_i^2}{2}+\mathrm{h}_i(\rho_i)\Big)u_i.
\end{equation}

We stress that the Baer-Nunziato system is only weakly hyperbolic and the assumptions in the introduction exclude resonance effects \cite{berthon_etal_12}, though the numerical experiments in section~\ref{sec:num_xp} will consider solutions close to resonance.

Note that given a smooth function $\psi(\alpha_2)$, combining both equations for the void fraction and partial density $\alpha_2\rho_2$, we get the following relation in conservation form

\begin{equation}\label{eq:generalized_alpha_eq_BN}
 \partial_t\big(\alpha_2\psi(\alpha_2)\rho_2\big) + \partial_x\big(\alpha_2\psi(\alpha_2)\rho_2u_2\big) = 0.
\end{equation}

Following the lines of Tadmor's proof of a minimum entropy principle for the gas dynamics equations \cite{tadmor86}, a maximum principle holds for the void fractions. This is summarized in the following lemma.

\begin{lemma}[maximum principle]\label{th:max_principle_BN}
The following estimates hold for solutions of the isentropic Baer-Nunziato model (\ref{eq:PDE_cons_noncons})-(\ref{eq:BN}):

\begin{equation}\label{eq:minmax_principle_BN}
 \underset{|x|\leq X+tu_2^{max}}{\mbox{ess inf}} \alpha_i^0(x) \leq \alpha_i(x,t) \leq \underset{|x|\leq X+tu_2^{max}}{\mbox{ess sup}} \alpha_i^0(x), \quad \mbox{for almost all }|x|\leq X, \; t>0,
\end{equation}
\noindent for $i=1,2$, where $u_2^{max}=\max_{\cal C}|u_2|$ over ${\cal C}=\{(x,\tau):\;|x|\leq X+(t-\tau)u_2^{max},\;0\leq\tau\leq t\}$ and $\alpha_i^0(\cdot)=\alpha_i(\cdot,0)$. 
\end{lemma}

\begin{proof}
Indeed, assuming first smooth solutions and integrating (\ref{eq:generalized_alpha_eq_BN}) over ${\cal C}$, we get

\begin{equation*}
 \int_{\partial{\cal C}} \alpha_2\psi(\alpha_2)\rho_2 (n_t+u_2n_x)ds = 0,
\end{equation*}
\noindent where $(n_t,n_x)$ denotes the unit normal pointing outward ${\cal C}$. Because $n_t+u_2n_x\geq0$ on $\partial{\cal C}$ for $0<\tau<t$ \cite[Lemma~3.1]{tadmor86}, we get for any smooth positive function $\psi(\alpha_2)$

\begin{equation*}
 \int_{|x|\leq X+tu_2^{max}} \alpha_2\psi(\alpha_2)\rho_2 dx \leq \int_{|x|\leq X} \alpha_2^0\psi(\alpha_2^0)\rho_2^0 dx.
\end{equation*}

Now using successively the positive functions $\psi(\alpha)=-\min(\alpha-\alpha^0,0)$ and $\psi(\alpha)=\max(\alpha-\alpha^0,0)$ and using vanishing viscosity arguments to obtain formal regularized versions of (\ref{eq:PDE_cons_noncons})-(\ref{eq:BN}) in the case of non-smooth solutions, we finally obtain (\ref{eq:minmax_principle_BN}) for $\alpha_2$. The same result also holds for $\alpha_1$ through the saturation condition (\ref{eq:saturation_condition}).\quad$\square$
\end{proof}

\subsubsection{Entropy conservative numerical fluxes}
The following fluxes are entropy conservative:

\begin{subeqnarray}\label{eq:BN_ECPC_fluxes}
 {\bf D}_{ec}^-({\bf u}^-,{\bf u}^+) &=& {\bf h}({\bf u}^-,{\bf u}^+) - {\bf f}({\bf u}^-) + {\bf d}^-({\bf u}^-,{\bf u}^+), \\
 {\bf D}_{ec}^+({\bf u}^-,{\bf u}^+) &=& {\bf f}({\bf u}^+) - {\bf h}({\bf u}^-,{\bf u}^+) + {\bf d}^+({\bf u}^-,{\bf u}^+),
\end{subeqnarray}
\noindent with 

\begin{equation}\label{eq:BN_ECPC_fluxes_detail}
{\bf h}({\bf u}^-,{\bf u}^+) = \begin{pmatrix} 0 \\ \ol{\alpha_1}\,\ol{u_1}\hat{h}_1(\rho_1^-,\rho_1^+) \\  \ol{\alpha_1}\big(\ol{u_1}^2\hat{h}_1(\rho_1^-,\rho_1^+) + \ol{\mathrm{p}_1}\big) \\ \ol{\alpha_2}\,\ol{u_2}\hat{h}_2(\rho_2^-,\rho_2^+) \\ \ol{\alpha_2}\big(\ol{u_2}^2\hat{h}_2(\rho_2^-,\rho_2^+) + \ol{\mathrm{p}_2}\big) \end{pmatrix}, \quad
{\bf d}^\pm({\bf u}^-,{\bf u}^+) = \frac{\du\alpha_1\df}{2} \begin{pmatrix} u_2^\pm\pm\beta_s \\ \pm\beta_s\hat{h}_1(\rho_1^-,\rho_1^+) \\ -\mathrm{p}_1^\pm\pm\beta_s\ol{u_1}\hat{h}_1(\rho_1^-,\rho_1^+) \\ \mp\beta_s\hat{h}_2(\rho_2^-,\rho_2^+) \\ \mathrm{p}_1^\pm\mp\beta_s\ol{u_2}\hat{h}_2(\rho_2^-,\rho_2^+) \end{pmatrix},
\end{equation}
\noindent where $\beta_s>0$ is a measure of the spectral radius of ${\bf A}({\bf u}_h)$ and will be evaluated in Lemma~\ref{th:pos_DG_BN}. The numerical fluxes for the partial densities in (\ref{eq:BN_ECPC_fluxes_detail}) read

\begin{equation}\label{eq:EC_density_flux}
 \hat{h}_i(\rho_i^-,\rho_i^+) = \left\{\begin{array}{ll}
 \frac{\du\mathrm{p}_i(\rho_i)\df}{\du\mathrm{h}_i(\rho_i)\df} & \mbox{if } \rho_i^- \neq \rho_i^+, \\ \rho_i & \mbox{if } \rho_i^- = \rho_i^+ = \rho_i,
\end{array}\right. \quad i=1,2.
\end{equation}

Indeed, inserting (\ref{eq:BN_entropy_pair}) into (\ref{eq:entropy_conserv_flux}) and using the Leibniz identities

\begin{equation}\label{eq:leibniz_id}
 \du\tfrac{u_i^2}{2}\df = \ol{u_i}\du u_i\df, \quad \du\alpha_i\mathrm{p}_iu_i\df = \ol{\alpha_i}(\ol{\mathrm{p}_i}\du u_i\df+\ol{u_i}\du \mathrm{p}_i\df)+\ol{\mathrm{p}_iu_i}\du\alpha_i\df, \quad i=1,2,
\end{equation}

\noindent we obtain

\begin{subeqnarray*}
(\etab'^-)^\top{\bf D}_{ec}^- &+& (\etab'^+)^\top{\bf D}_{ec}^+ - \du q\df \overset{(\ref{eq:BN_ECPC_fluxes})}{=} \du\etab'^\top({\bf f}-{\bf h})\df + 2\ol{\etab'^\top{\bf d}} - \du q\df \\
  &\overset{(\ref{eq:leibniz_id})}{=}& -\ol{\alpha_1}\,\ol{u_1}\hat{h}_1\du\mathrm{h}_1-\cancel{\tfrac{u_1^2}{2}}\df - \ol{\alpha_1}(\cancel{\ol{u_1}^2\hat{h}_1} + \ol{\mathrm{p}_1})\du u_1 \df + \du\cancel{\alpha_1\rho_1u_1(\mathrm{h}_1-\tfrac{u_1^2}{2})}\df + \du\alpha_1(\cancel{\rho_1u_1^2}+\mathrm{p}_1)u_1\df \\
 &-& \ol{\alpha_2}\,\ol{u_2}\hat{h}_2\du\mathrm{h}_2-\cancel{\tfrac{u_2^2}{2}}\df - \ol{\alpha_2}(\cancel{\ol{u_2}^2\hat{h}_2} + \ol{\mathrm{p}_2})\du u_2 \df + \du\cancel{\alpha_2\rho_2u_2(\mathrm{h}_2-\tfrac{u_2^2}{2})}\df + \du\alpha_2(\cancel{\rho_2u_2^2}+\mathrm{p}_2)u_2\df \\
 &-& \du \cancel{\alpha_1\rho_1(\tfrac{u_1^2}{2}+\mathrm{h}_1)u_1} + \cancel{\alpha_2\rho_2(\tfrac{u_2^2}{2}+\mathrm{h}_2)u_2}\df \\
 &+& \du\alpha_1\df\Big[ \ol{u_2(\mathrm{p}_2-\mathrm{p}_1)} - \ol{u_1\mathrm{p}_1} + \ol{u_2\mathrm{p}_1} + \tfrac{\beta_s}{2}\Big(\du\mathrm{p}_2-\mathrm{p}_1\df + \hat{h}_1\big(\du\mathrm{h}_1-\tfrac{u_1^2}{2}\df + \ol{u_1}\du u_1\df) \\
 &-&  \hat{h}_2 \big(\du\mathrm{h}_2-\tfrac{u_2^2}{2}\df + \ol{u_2}\du u_2\df\big)\Big)\Big] \\
  &\overset{(\ref{eq:EC_density_flux})}{=}& -\ol{\alpha_1}(\ol{u_1}\du\mathrm{p}_1\df+\ol{\mathrm{p}_1}\du u_1\df) + \du\alpha_1u_1\mathrm{p}_1\df - \ol{u_1\mathrm{p}_1}\du\alpha_1\df \\
 &-&\ol{\alpha_2}(\ol{u_2}\du\mathrm{p}_2\df+\ol{\mathrm{p}_2}\du u_2\df) + \du\alpha_2u_2\mathrm{p}_2\df - \ol{u_2\mathrm{p}_2}\du\alpha_2\df \\
  &\overset{(\ref{eq:leibniz_id})}{=}& 0.
\end{subeqnarray*}

Some remarks are in order. The numerical conservation flux ${\bf h}(\cdot,\cdot)$ in (\ref{eq:BN_ECPC_fluxes}) is symmetric, consistent and differentiable, while the fluctuation fluxes have the form (\ref{eq:roe_type_fluxes}a) with ${\bf\cal A}({\bf u}^-,{\bf u}^+) = (\ol{u_2},0,-\ol{\mathrm{p}_1},0,\ol{\mathrm{p}_1})^\top$ and therefore satisfy (\ref{eq:roe_type_fluxes}c,d) and are path-conservative (\ref{eq:castro_path_cons_flux}) for a linear path in $u_2$ and $\mathrm{p}_1$. Due to the presence of the nonlinear fluxes $\hat{h}_i$, the ${\bf d}^\pm$ are examples of fluctuations fluxes in non-splitting form.  Finally, the DGSEM with the fluxes (\ref{eq:BN_ECPC_fluxes}) is by construction conservative for the mixture density and momentum.

%
%
\section{High-order DGSEM for the isentropic Baer-Nunziato model}\label{sec:BN}

\subsection{Entropy stable fluxes}

We now focus on the design of a positive and entropy stable DG scheme for the two-pressure two-velocity isentropic model (\ref{eq:PDE_cons_noncons}) with (\ref{eq:BN}). For that purpose, we introduce the fully discrete scheme for a one-step first-order explicit time discretization and analyze its properties. High-order time integration will be done by using strong-stability preserving explicit Runge-Kutta methods \cite{shu-osher88} that keep the properties of the first-order in time scheme.

Let $t^{(n)}=n\Delta t$, with $\Delta t>0$ the time step, set $\lambda=\tfrac{\Delta t}{h}$, and use the notations ${\bf u}_h^{(n)}(\cdot)={\bf u}_h(\cdot,t^{(n)})$ and ${\bf U}_j^{k,n}={\bf U}_j^k(t^{(n)})$. The DGSEM scheme for solving the isentropic Baer-Nunziato equations reads

\begin{equation}\label{eq:discr_ECPC_DGSEM_BN}
 \frac{\omega_kh}{2}\frac{{\bf U}_j^{k,n+1}-{\bf U}_j^{k,n}}{\Delta t} + {\bf R}_j^k({\bf u}_h^{(n)}) = 0,
\end{equation}
\noindent with ${\bf R}_j^k(\cdot)$ defined in (\ref{eq:discr_ECPC_DGSEM_BN_res}) and where the entropy conservative fluxes (\ref{eq:BN_ECPC_fluxes}) are used in the definition of (\ref{eq:ECPC_intvol_func}). We follow the strategy in \cite{castro_etal_13} to design entropy stable fluxes at interfaces:

\begin{equation}\label{eq:discr_ECPC_DGSEM_BN_ES_fluxes}
{\bf D}^\pm({\bf u}^-,{\bf u}^+) = {\bf D}_{ec}^-({\bf u}^-,{\bf u}^+) \pm \epsilon_v\beta_s{\bf D}_v({\bf u}^-,{\bf u}^+)\du\etab'({\bf u})\df,
\end{equation}
\noindent with $\epsilon_v>0$ and the positive diagonal matrix

\begin{equation}
 {\bf D}_v({\bf u}^-,{\bf u}^+) = \mbox{diag}\big(0,\tfrac{\ol{\alpha_1\rho_1}}{\ol{u_1}^2+\ol{c_1}^2},\ol{\alpha_1\rho_1},\tfrac{\ol{\alpha_1\rho_1}}{\ol{u_1}^2+\ol{c_1}^2},\ol{\alpha_2\rho_2}\big).
\end{equation}

\subsection{Properties of the discrete scheme}\label{sec:ana_1st_order_time_discr}

We have the following results that guaranty positivity of the solution and the maximum principle (\ref{eq:minmax_principle_BN}) for the fully discrete solution of the DGSEM.

\begin{theorem}\label{th:pos_DG_BN}
 Assume that $\rho_{i,j\in\mathbb{Z}}^{0\leq k\leq p,n}>0$ and $\alpha_{i,j\in\mathbb{Z}}^{0\leq k\leq p,n}>0$ for $i=1,2$, then under the CFL condition

\begin{equation}\label{eq:CFL_cond_iso_BN}
 \lambda\max_{j\in\mathbb{Z}}\max_{0\leq k\leq p}\frac{1}{\omega_k}\Big(\big\langle u_{2,h}^{(n)},d_x\phi_j^k\big\rangle_j^p+\delta_{k,p}\frac{\beta_s-u_{2,j}^{p,n}}{2}+\delta_{k,0}\frac{\beta_s+u_{2,j}^{0,n}}{2}\Big) < \frac{1}{2},
\end{equation}
\noindent we have for the cell averages at time $t^{(n+1)}$

\begin{equation}\label{eq:pos_rho_alpha}
 \langle\alpha_i\rho_{i}\rangle_{j}^{(n+1)}>0, \quad \langle\alpha_i\rangle_{j}^{(n+1)}>0, \quad i=1,2, \quad j\in\mathbb{Z},
\end{equation}
\noindent and 
\begin{eqnarray}\label{eq:convex_comb_alpha}
 \langle\alpha_1\rangle_j^{(n+1)} &=& \sum_{k=0}^p\bigg(\frac{\omega_k}{2}-\lambda\Big(\big\langle u_{2,h}^{(n)},d_x\phi_j^k\big\rangle_j^p+\delta_{k,p}\frac{\beta_s-u_{2,j}^{p,n}}{2}+\delta_{k,0}\frac{\beta_s+u_{2,j}^{0,n}}{2}\Big)\bigg)\alpha_{1,j}^{k,n} \nonumber \\ 
 && +\lambda\frac{\beta_s-u_{2,j}^{p,n}}{2}\alpha_{1,j+1}^{0,n} + \lambda\frac{\beta_s+u_{2,j}^{0,n}}{2}\alpha_{1,j-1}^{p,n}, \quad j\in\mathbb{Z},
\end{eqnarray}
\noindent is a convex combination of DOFs at time $t^{(n)}$ where

\begin{equation*}
 \beta_s = \max\Big(|u_{i,j-1}^{p,n}|+c_{i,j-1}^{p,n},|u_{i,j}^{0\leq k\leq p,n}|+c_{i,j}^{0\leq k\leq p,n},|u_{i,j+1}^{0,n}|+c_{i,j+1}^{0,n}:\;i=1,2\Big),
\end{equation*}
\noindent and $c_i=c_i(\rho_i)$ denotes the speed of sound of phase $i$.
\end{theorem}

\begin{proof}
The positivity of the cell averaged partial densities rely on the techniques introduced in \cite{perthame_shu_96,zhang_shu_10b} to rewrite a conservative high-order scheme for the mean value as a convex combination of positive first-order schemes. Summing the first component in (\ref{eq:discr_ECPC_DGSEM_BN}) over $0\leq k\leq p$ gives (\ref{eq:convex_comb_alpha}) and it is direct to check that it is a convex combination under the condition (\ref{eq:CFL_cond_iso_BN}).\quad$\square$
\end{proof}

The following result is useful to prevent spurious oscillations in the numerical solution. Indeed, the present schemes satisfies the Abgrall's criterion \cite{abgrall_96} that states that uniform velocity and pressure must remain uniform at all time.

\begin{lemma}[Abgrall's criterion]\label{th:unif_u_p_DG_BN}
Assume that the velocity and pressure are uniform and equal at time $t^{(n)}$:

\begin{equation}\label{eq:unif_u_p_DG_BN}
 u_{i,j}^{k,n}=u, \quad \mathrm{p}_{i,j}^{k,n}=\mathrm{p}, \quad i=1,2, \quad \forall j\in\mathbb{Z}, \quad 0\leq k\leq p,
\end{equation}
\noindent then they remain uniform and equal at time $t^{(n+1)}$.
\end{lemma}

\begin{proof}
The assumption (\ref{eq:unif_u_p_DG_BN}) on pressures require uniform densities $\rho_{i,j}^{k,n}=\rho_i$, $i=1,2$, so $\du\etab'({\bf u})\df=0$. Then, the entropy conservative fluxes (\ref{eq:BN_ECPC_fluxes}) and entropy stable fluxes (\ref{eq:discr_ECPC_DGSEM_BN_ES_fluxes}) reduce to

\begin{equation*}
 {\bf D}^\pm({\bf u}^-,{\bf u}^+) = {\bf D}_{ec}^\pm({\bf u}^-,{\bf u}^+) = \frac{\du\alpha_1\df}{2} \begin{pmatrix} u\pm\beta_s \\ \rho_1(u\pm\beta_s) \\ \rho_1u(u\pm\beta_s) \\ -\rho_2(u\pm\beta_s) \\ -\rho_2u(u\pm\beta_s) \end{pmatrix},
\end{equation*}

\noindent so the explicit residuals in (\ref{eq:discr_ECPC_DGSEM_BN}) become ${\bf R}_j^k({\bf u}_h)=R_{1,j}^{k}(1,\rho_1,\rho_1u,-\rho_2,-\rho_2u)^\top$ with

\begin{subeqnarray}\label{eq:discr_ECPC_DGSEM_res_u_p_unif}
R_{1,j}^{k} &=& \omega_ku\sum_{l=0}^p\alpha_{1,j}^{k}D_{kl}+\tfrac{\delta_{kp}(u-\beta_s)}{2}\du\alpha_1\df_{j+\frac{1}{2}}+\tfrac{\delta_{k0}(u+\beta_s)}{2}\du\alpha_1\df_{j-\frac{1}{2}},\\
 &\overset{(\ref{eq:saturation_condition})}{\underset{(\ref{eq:interp_lag_unite_deriv})}{=}}& -\omega_ku\sum_{l=0}^p\alpha_{2,j}^{k}D_{kl}-\tfrac{\delta_{kp}(u-\beta_s)}{2}\du\alpha_2\df_{j+\frac{1}{2}}-\tfrac{\delta_{k0}(u+\beta_s)}{2}\du\alpha_2\df_{j-\frac{1}{2}}.
\end{subeqnarray}
We thus rewrite (\ref{eq:discr_ECPC_DGSEM_BN}) as

\begin{subeqnarray}\label{eq:iscr_ECPC_DGSEM_BN_u_p_unif}
\tfrac{\omega_kh}{2}\tfrac{\alpha_{1,j}^{k,n+1}-\alpha_{1,j}^{k,n}}{\Delta t} + R_{1,j}^{k,n} &=& 0,\\
\tfrac{\omega_kh}{2}\tfrac{\alpha_{1,j}^{k,n+1}(\rho_{1,j}^{k,n+1}-\rho_1)+\rho_1(\alpha_{1,j}^{k,n+1}-\alpha_{1,j}^{k,n})}{\Delta t} + \rho_1R_{1,j}^{k,n} &=& 0,\\
\tfrac{\omega_kh}{2}\tfrac{(\alpha_{1}\rho_1)_{j}^{k,n+1}(u_{1,j}^{k,n+1}-u)+u\big((\alpha_{1}\rho_1)^{k,n+1}-\rho_1\alpha_{1,j}^{k,n}\big)}{\Delta t} + \rho_1uR_{1,j}^{k,n} &=& 0,\\
\tfrac{\omega_kh}{2}\tfrac{\alpha_{2,j}^{k,n+1}(\rho_{2,j}^{k,n+1}-\rho_2)+\rho_2(\alpha_{2,j}^{k,n+1}-\alpha_{2,j}^{k,n})}{\Delta t} - \rho_2R_{1,j}^{k,n} &=& 0,\\
\tfrac{\omega_kh}{2}\tfrac{(\alpha_{2}\rho_2)_{j}^{k,n+1}(u_{2,j}^{k,n+1}-u)+u\big((\alpha_{2}\rho_2)^{k,n+1}-\rho_2\alpha_{2,j}^{k,n}\big)}{\Delta t} - \rho_2uR_{1,j}^{k,n} &=& 0.
\end{subeqnarray}
Then, (\ref{eq:iscr_ECPC_DGSEM_BN_u_p_unif}b)$-\rho_1$(\ref{eq:iscr_ECPC_DGSEM_BN_u_p_unif}a) implies $\rho_{1,j}^{k,n+1}=\rho_1$, while (\ref{eq:iscr_ECPC_DGSEM_BN_u_p_unif}c)$-\rho_1u$(\ref{eq:iscr_ECPC_DGSEM_BN_u_p_unif}a) gives $u_{1,j}^{k,n+1}=u$. Then, using (\ref{eq:discr_ECPC_DGSEM_res_u_p_unif}b) and (\ref{eq:iscr_ECPC_DGSEM_BN_u_p_unif}d,e) we obtain $\rho_{2,j}^{k,n+1}=\rho_2$ and $u_{2,j}^{k,n+1}=u$.\quad$\square$
\end{proof}
\subsection{Limiting strategy}\label{sec:limiters}

The properties in Theorem~\ref{th:pos_DG_BN} hold only for the cell averaged value of the numerical solution at time $t^{(n+1)}$, which is not sufficient for robustness and stability of numerical computations. However, these results motivate the use of a posteriori limiters introduced in \cite{zhang_shu_10a,zhang_shu_10b}. These limiters aim at extending preservation of invariant domains \cite{zhang_shu_10b} or maximum principle  \cite{zhang_shu_10a} from mean values to nodal values within elements. 

We enforce positivity of nodal values of partial densities and the maximum principle (\ref{eq:minmax_principle_BN}) by using the linear limiter

\begin{equation}\label{eq:pos_limiter}
 \tilde{\bf U}_j^{k,n+1} = \theta_j({\bf U}_j^{k,n+1}- \langle{\bf u}\rangle_j^{(n+1)}) + \langle{\bf u}\rangle_j^{(n+1)}, \quad 0\leq k\leq p, \quad j \in \mathbb{Z},
\end{equation}
\noindent with  $0\leq \theta_j\leq1$ defined by $\theta_j:=\min(\theta_j^{\rho_i},\theta_j^{\alpha_i}:i=1,2)$ where

\begin{eqnarray*}
 \theta_j^{\rho_i} &=& \min\Big(\frac{\langle\rho_i\rangle_{j}^{(n+1)}-\epsilon}{\langle\rho_i\rangle_{j}^{(n+1)}-\rho_{i,j}^{min}},1\Big), \quad \rho_{i,j}^{min}=\min_{0\leq k\leq p} \rho_{i,j}^{k,n+1}, \\
 \theta_j^{\alpha_i} &=& \min\Big(\frac{\langle\alpha_i\rangle_{j}^{(n+1)}-m_{j}^{\alpha_i}}{\langle\alpha_i\rangle_{j}^{(n+1)}-\alpha_{i,j}^{min}},\frac{\langle\alpha_i\rangle_{j}^{(n+1)}-M_{j}^{\alpha_i}}{\langle\alpha_i\rangle_{j}^{(n+1)}-\alpha_{i,j}^{max}},1\Big),  \quad \alpha_{i,j}^{min/max}=\underset{0\leq k\leq p}{\min/\max}\; \alpha_{i,j}^{k,n+1},
\end{eqnarray*}
\noindent $0<\epsilon\ll1$ a parameter, and

\begin{equation*}
 m_{j}^{\alpha_i} = \min\big(\alpha_{i,j-1}^{p,n},\alpha_{i,j}^{0\leq k\leq p,n},\alpha_{i,j+1}^{0,n}\big), \quad M_{j}^{\alpha_i} = \max\big(\alpha_{i,j-1}^{p,n},\alpha_{i,j}^{0\leq k\leq p,n},\alpha_{i,j+1}^{0,n}\big).
\end{equation*}

The limiter (\ref{eq:pos_limiter}) guaranties a discrete maximum principle on the void fractions $m_{j}^{\alpha_i} \leq\tilde\alpha_{i,j}^{0\leq k\leq p,n+1} \leq M_{j}^{\alpha_i}$ and keeps the entropy inequality (\ref{eq:entropy_ineq_cont}) at the discrete level in the sense that \cite[Lemma~3.1]{chen_shu_17} for $\eta$ convex we have

\begin{equation*}
 \langle\tilde\eta\rangle_j^{(n+1)}:=\sum_{k=0}^p\frac{\omega_k}{2}\eta(\tilde{\bf U}_j^{k,n+1}) \leq \langle\eta\rangle_j^{(n+1)}.
\end{equation*}

Likewise, phase densities and velocities remain unchanged by the limiter (\ref{eq:pos_limiter}) so uniform velocity and pressure profiles are conserved.

%
%
\section{Numerical experiments}\label{sec:num_xp}

In the following, we consider Riemann problems for nonconservative systems associated to initial conditions 

\begin{equation*}
 {\bf u}_0(x) = \left\{ \begin{array}{rl}  {\bf u}_L, & x<0, \\ {\bf u}_R, & x>0. \end{array} \right.
\end{equation*}

The set of initial conditions is given in Table~\ref{tab:RP_IC}. Figures~\ref{fig:solution_PR_noncons_2} to \ref{fig:solution_PR_BN} compare the numerical solution in symbols with the exact solution in lines. Problems RP1 and RP2 come from \cite{coquel_etal_13}, while RP3 is adapted from \cite{munkejord_07}.

\begin{table}
     \begin{center}
     \caption{Initial conditions and physical parameters of Riemann problems with ${\bf\cal U}=(u,v)^\top$ for the $2\times2$ system (\ref{eq:Sys22_LD_noncons}) and ${\bf\cal U}=(\alpha_1,\rho_1,u_1,\rho_2,u_2)^\top$ for the isentropic Baer-Nunziato system (\ref{eq:PDE_cons_noncons})-(\ref{eq:BN}).}
     \begin{tabular}{llccc}
        \noalign{\smallskip}\hline\noalign{\smallskip}
        test & model & left state ${\bf\cal U}_L$ & right state ${\bf\cal U}_R$ & $t$ \\
        \noalign{\smallskip}\hline\noalign{\smallskip}
        RP0 & (\ref{eq:Sys22_LD_noncons}) & $\big(3,\tfrac{1}{2}\big)^\top$ & $\big(\tfrac{3}{4},1\big)^\top$ & $0.15$ \\
        \noalign{\smallskip}\hline\noalign{\smallskip}
        RP1 & (\ref{eq:BN}) & $\left(\begin{array}{l} 0.1\\0.85\\0.4609513139\\0.96\\0.0839315299 \end{array}\right)$ & $\left(\begin{array}{l} 0.6\\1.2520240113\\0.7170741165\\0.2505659851\\-0.3764790609 \end{array}\right)$ & $0.14$ \\
        RP2 & (\ref{eq:BN}) & $\left(\begin{array}{l} 0.999\\ 1.8\\ 0.747051068928543\\ 3.979765198025580\\ 0.6 \end{array}\right)$ & $\left(\begin{array}{l} 0.4\\ 2.081142099494683\\ 0.267119045902047\\ 5.173694757433254\\ 1.069067604724276 \end{array}\right)$ & $0.1$ \\
        RP3 & (\ref{eq:BN}) & $\left(\begin{array}{l} 0.29\\2.0059425069187893 \\ 65\\ 2.0059425069187893\\ 1  \end{array}\right)$ & $\left(\begin{array}{l} 0.3\\2.0059425069187893 \\ 50\\2.0059425069187893\\ 1 \end{array}\right)$ & $0.08$ \\
        \noalign{\smallskip}\hline\noalign{\smallskip}
    \end{tabular}
     \label{tab:RP_IC}
    \end{center}
\end{table}

For the time integration, we use the three stage third-order strong-stability preserving Runge-Kutta time integration scheme of Shu and Osher \cite{shu-osher88}. We evaluate the time step with a safety factor of $\Delta t=0.9\times\lambda h$, where $\lambda$ is evaluated from 

\begin{equation*}
 \lambda\max_{j\in\mathbb{Z}}\max_{0\leq k\leq p}\big(|U_j^{k,n}+V_j^{k,n}|,|V_j^{k,n}|,\tfrac{2\epsilon_vh}{2p+1}\big) \leq 1,
\end{equation*}
\noindent for the $2\times2$ system (\ref{eq:Sys22_LD_noncons}) and from (\ref{eq:CFL_cond_iso_BN}) for the isentropic Baer-Nunziato model.

In both cases, entropy stable schemes at element interfaces are obtained by adding viscosity operators that mimic, at the discrete level, physical parabolic regularizations in the same way as done in \cite{castro_etal_13}. Let us stress that we here consider systems having nonconservative products associated with LD characteristic fields for which finite difference schemes have been shown to converge to the physically relevant solution \cite{castro_etal_08}. However the present strategy may fail for strong shocks where the agreement between regularizations at discrete and continuous levels may not be satisfied \cite{castro_etal_08}.

\subsection{Nonconservative product associated to a LD field}

Figure~\ref{fig:solution_PR_noncons_2} shows results for a $1$-shock, $2$-contact problem (RP0 in Table~\ref{tab:RP_IC}) for system (\ref{eq:Sys22_LD_noncons}). We compare solutions obtained with the entropy stable scheme (\ref{eq:semi-discr_ECPC_DGSEM}), with (\ref{eq:ECPC_intvol_func}) evaluated from the entropy conservative fluxes (\ref{eq:2_2_sys_EC_flux}), or with the original DGSEM (\ref{eq:semi-discr_DGSEM}). In both cases, we use the same entropy stable numerical fluxes (\ref{eq:2_2_sys_ES_flux}) at interfaces. The results highlight the importance of the modification of the volume integral in (\ref{eq:semi-discr_ECPC_DGSEM}) to satisfy the entropy inequality. The second order solution without this modification does not tend to the exact weak solution and contains non-physical waves even when the mesh is refined. We note that higher-order computations for $p\geq2$ without the correction (\ref{eq:ECPC_intvol_func}) were seen to blow up due to a change of sign in the $u$ component of the solution which induce a loss of strict hyperbolicity of system (\ref{eq:Sys22_LD_noncons}). The correction (\ref{eq:ECPC_intvol_func}) successfully stabilizes the computation and the numerical solution now tends to the exact entropy solution. 

\begin{figure}
\begin{center}
\subfigure[$p=1$, $N=250$]{\epsfig{figure=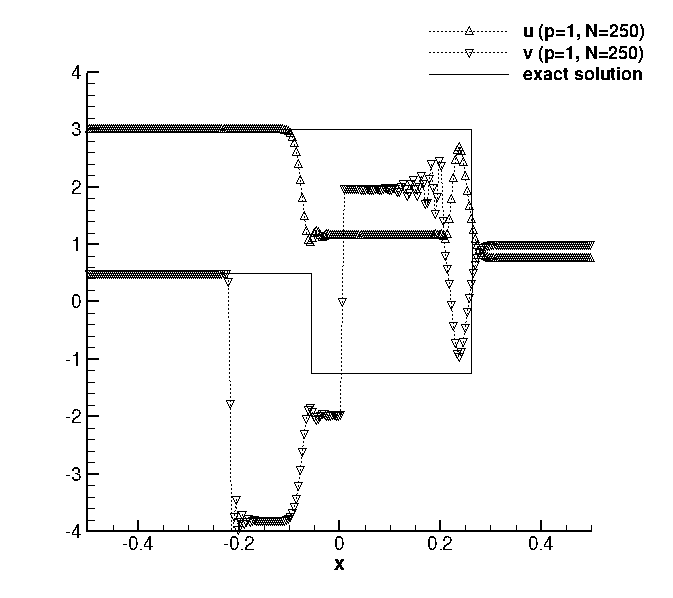,width=6cm}}
\subfigure[$p=1$, $N=2500$]{\epsfig{figure=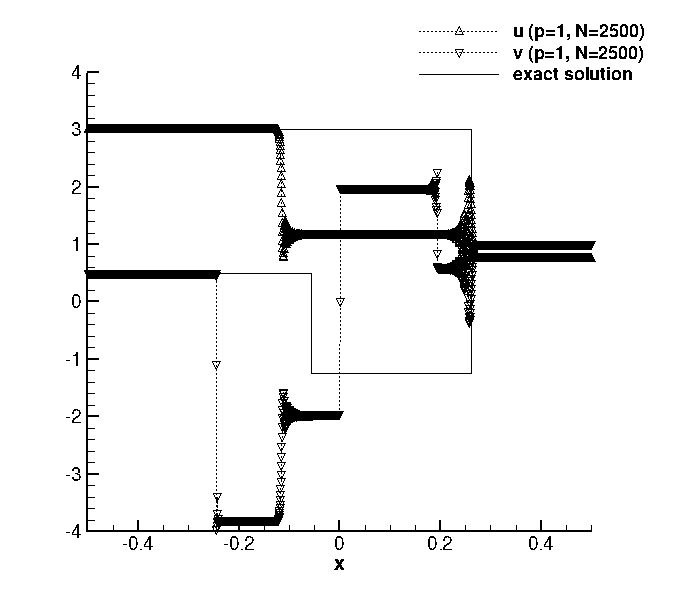,width=6cm}}\\
\subfigure[$p=1$, $N=250$ (ES)]{\epsfig{figure=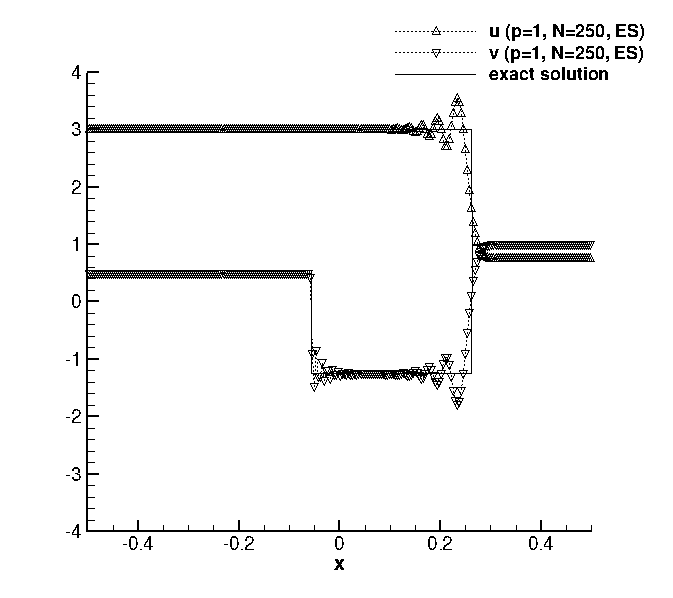,width=6cm}}
\subfigure[$p=4$, $N=100$ (ES)]{\epsfig{figure=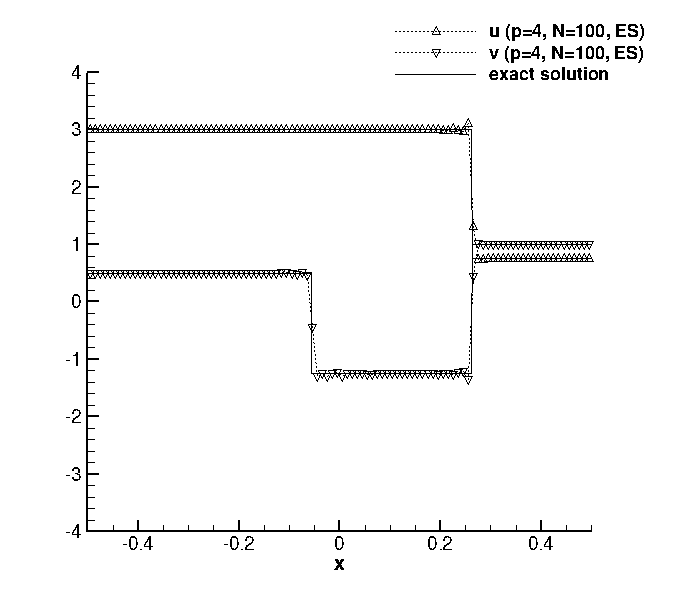,width=6cm}}
\caption{$2\times2$ system: RP0 discretized with polynomial degree $p$, $N$ cells and entropy stable (ES) modification (\ref{eq:ECPC_intvol_func}) or not.}
\label{fig:solution_PR_noncons_2}
\end{center}
\end{figure}

\subsection{Isentropic Baer-Nunziato model}

For the numerical experiments on the isentropic Baer-Nunziato model (\ref{eq:PDE_cons_noncons})-(\ref{eq:BN}), we consider polytropic ideal gas with equations of state of the form $\mathrm{p}_i(\rho_i)=\kappa\rho_i^{\gamma_i}$ with $\kappa>0$ and $\gamma_i>1$, $i=1,2$. Computations are done with the entropy stable numerical scheme (\ref{eq:discr_ECPC_DGSEM_BN}) and fourth-order accuracy, $p=3$. The limiter (\ref{eq:pos_limiter}) is applied at the end of each stage unless stated otherwise.

We first consider the advection of a discontinuity of the void fraction in uniform velocities, $u_{1,0}=u_{2,0}=1$ and pressures, $\mathrm{p}_{1,0}=\mathrm{p}_{2,0}=1$, so the mass and momentum equations in (\ref{eq:PDE_cons_noncons})-(\ref{eq:BN}) are trivially satisfied. The pressure law parameters are $\kappa=1$, $\gamma_1=1.4$, and $\gamma_2=1.2$. Figure~\ref{fig:solution_u_p_uniform_BN} presents the solution obtained at time $t=0.1$ with and without limiter. In both cases, the velocities and pressures remain uniform as expected from Lemma~\ref{th:unif_u_p_DG_BN}, but the limiter is seen to introduce numerical dissipation that smears the contact discontinuity. The design of a sharp limiter would help to improve the solution but is beyond the scope of the present study where we rather focus on stability and robustness issues.

\begin{figure}
\begin{center}
\subfigure[no limiter]{\epsfig{figure=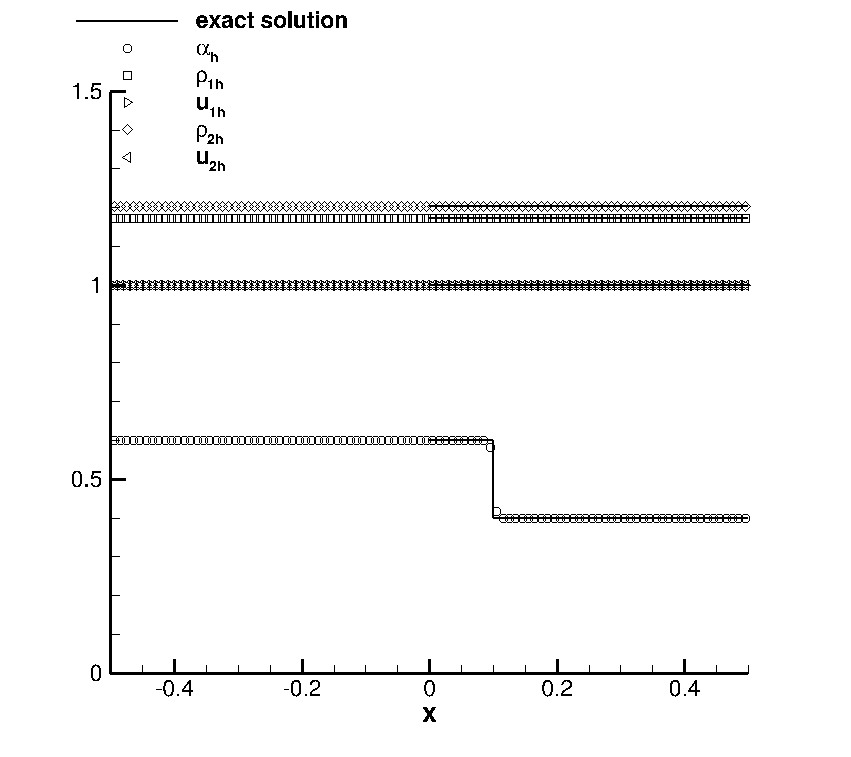,width=6cm}}
\subfigure[limiter (\ref{eq:pos_limiter})]{\epsfig{figure=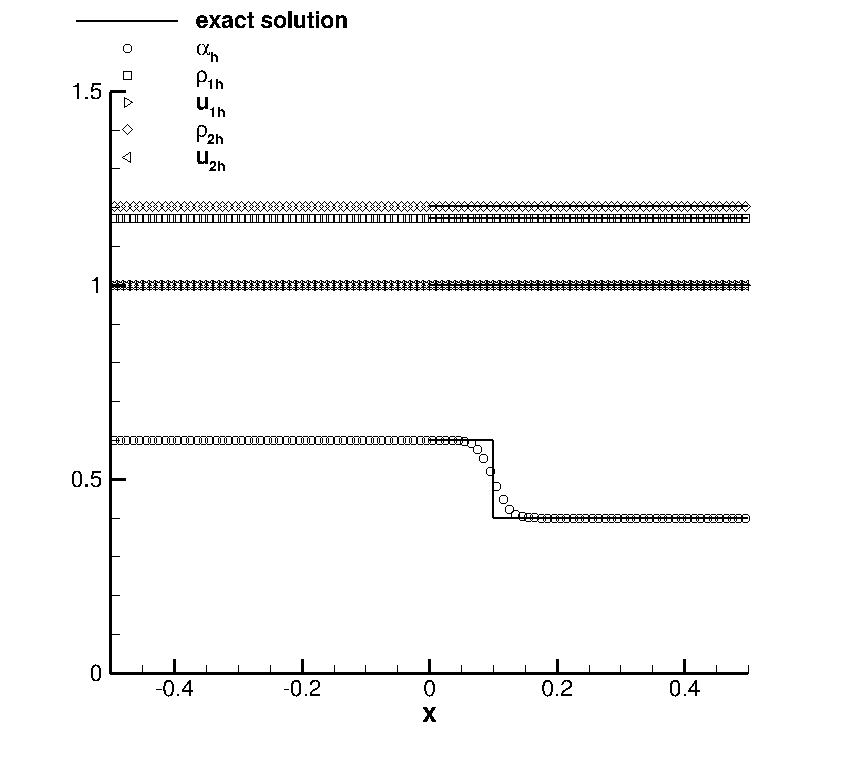,width=6cm}}
\caption{Isentropic Baer-Nunziato model: advection of a void fraction discontinuity discretized with polynomial degree $p=3$, $N=100$ cells and entropy stable scheme.}
\label{fig:solution_u_p_uniform_BN}
\end{center}
\end{figure}

Figure~\ref{fig:solution_PR_BN} presents the solution of Riemann problems associated to the initial conditions of Table~\ref{tab:RP_IC}. For RP1 and RP2, we use $\kappa=1$, $\gamma_1=3$ and $\gamma_2=1.5$. RP2 considers solutions close to resonance with a vanishing phase $2$ where $\alpha_2=10^{-3}$ and where the contact discontinuity separates a mixture region where the two phases coexist from a single phase region. The shock and rarefaction waves are well captured, while the contact wave is slightly diffused as an effect of the limiter as observed in the precedent experiment. RP3 is adapted from the experiment with large relative velocity for one pressure models in \cite{munkejord_07} and we set $\kappa=10^5$ and $\gamma_1=\gamma_2=1.4$. Spurious oscillations of low amplitude are observed in the neighborhood of the strong shocks, but the results are in good quantitative agreement with the exact solution. We stress that our experiments show that the correction (\ref{eq:ECPC_intvol_func}) of the volume integral is strongly needed for stabilizing the computations which would blow up otherwise.

\begin{figure}
\begin{bigcenter}
\subfigure{\begin{picture}(0,0) \put(-6,50){$\alpha_1$} \end{picture}}
\subfigure{\epsfig{figure=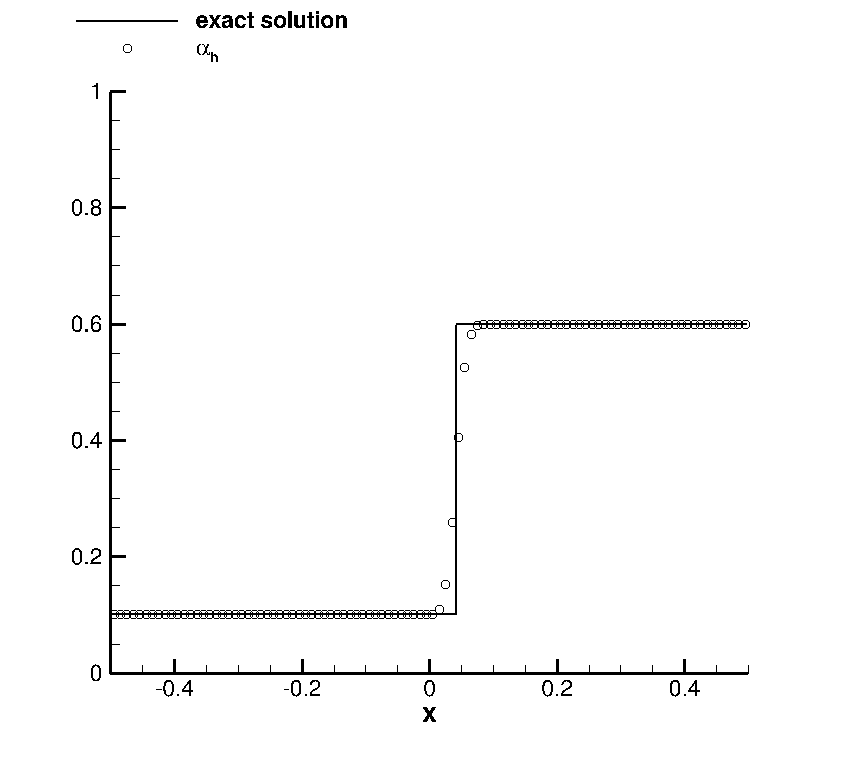,width=4cm}}
\subfigure{\epsfig{figure=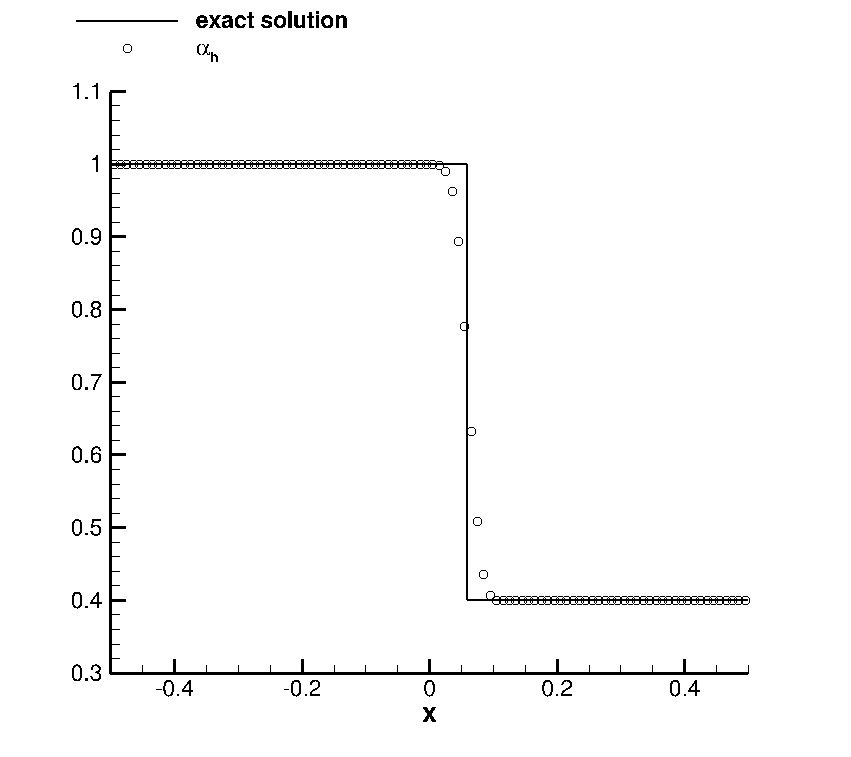,width=4cm}}
\subfigure{\epsfig{figure=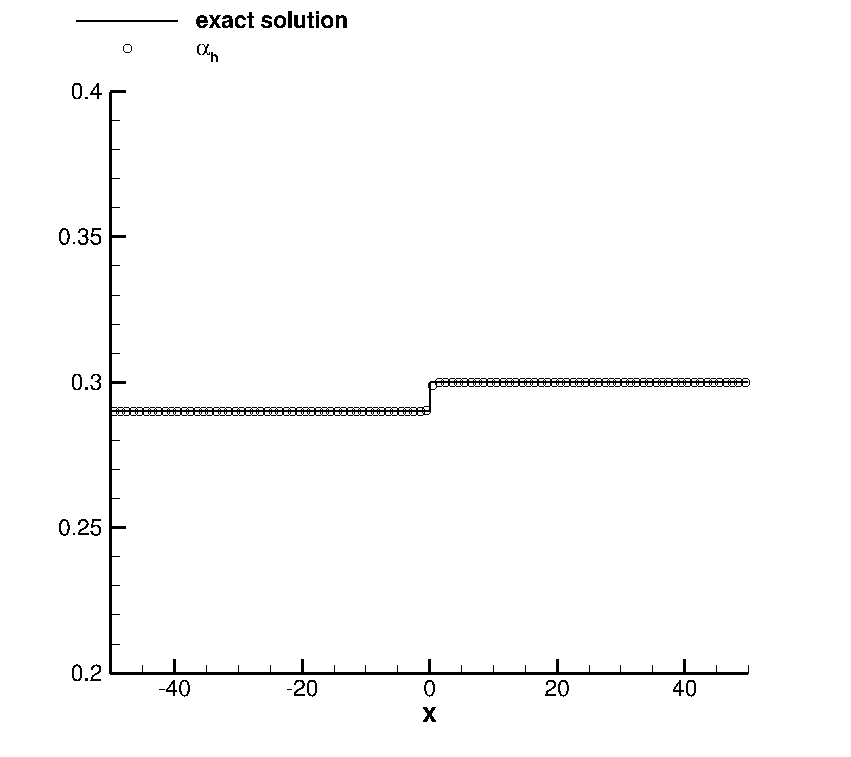,width=4cm}}\\
\subfigure{\begin{picture}(0,0) \put(-6,50){$\rho_1$} \end{picture}}
\subfigure{\epsfig{figure=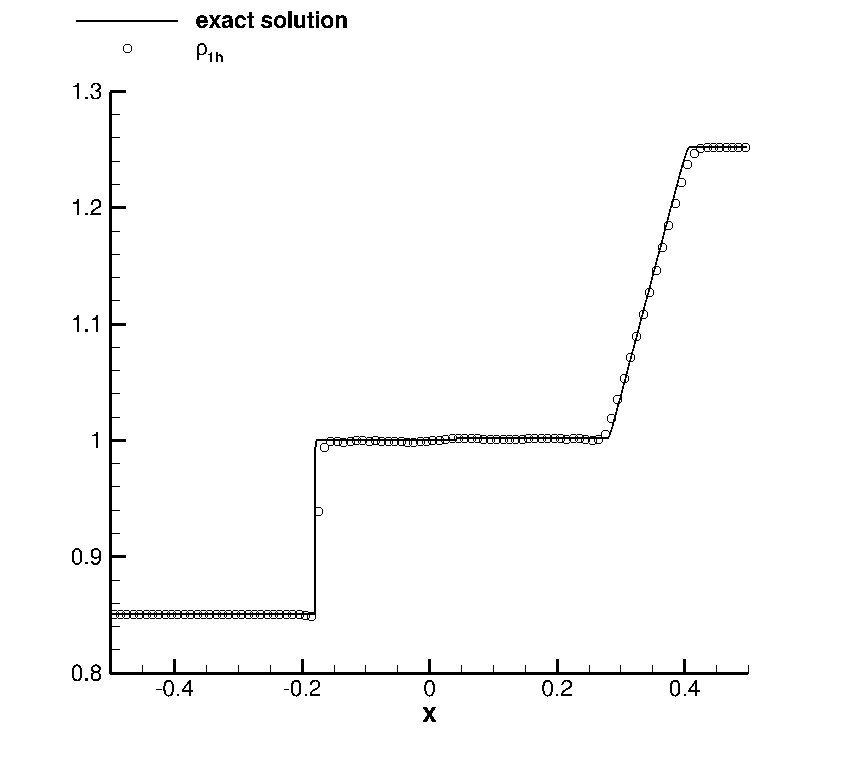 ,width=4cm}}
\subfigure{\epsfig{figure=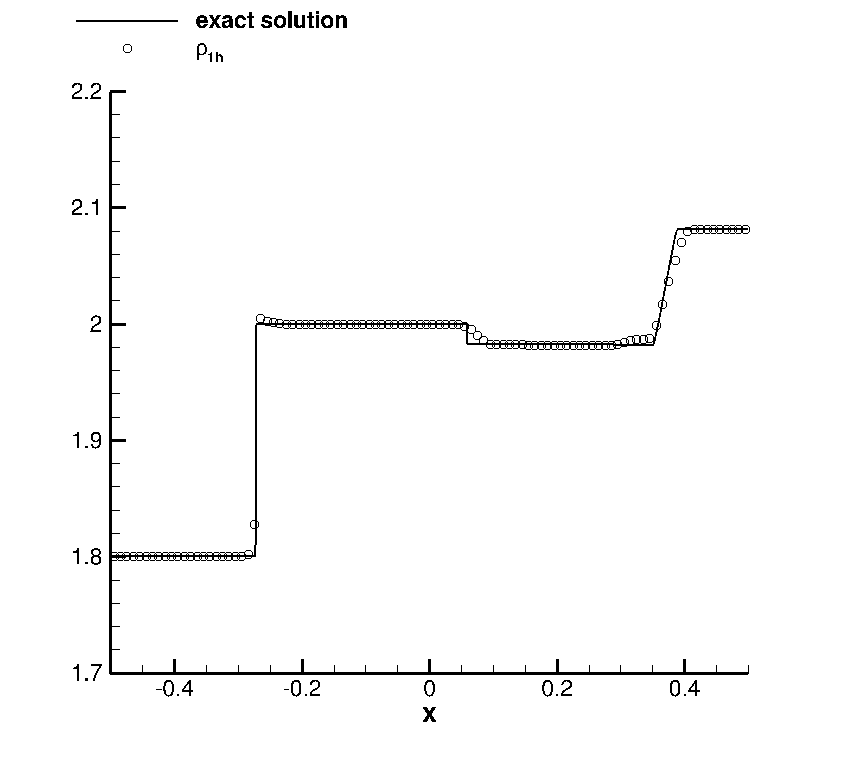 ,width=4cm}}
\subfigure{\epsfig{figure=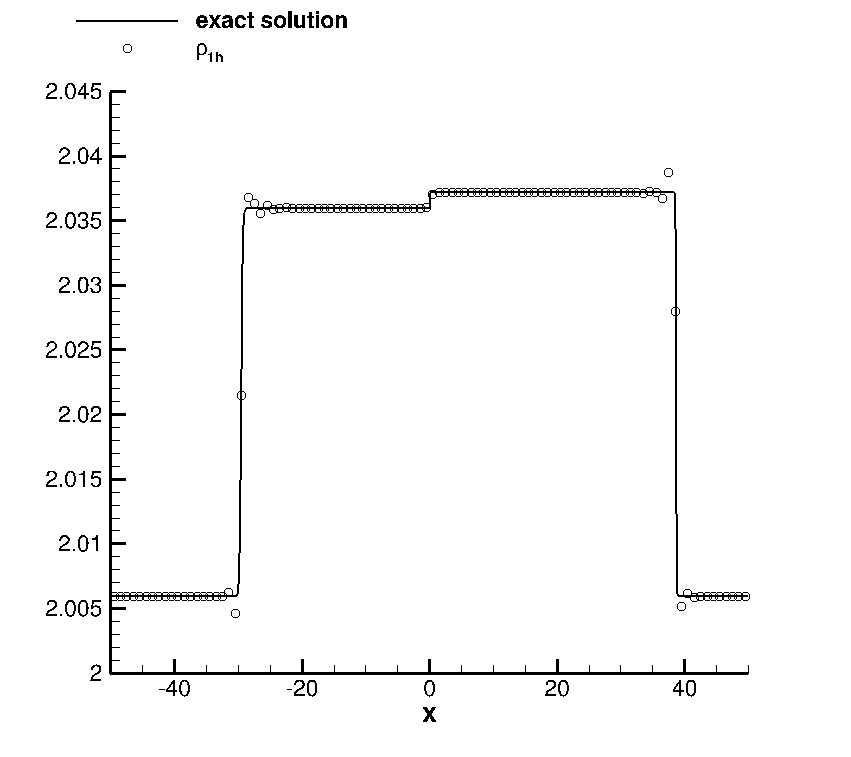 ,width=4cm}}\\
\subfigure{\begin{picture}(0,0) \put(-6,50){$u_1$} \end{picture}}
\subfigure{\epsfig{figure=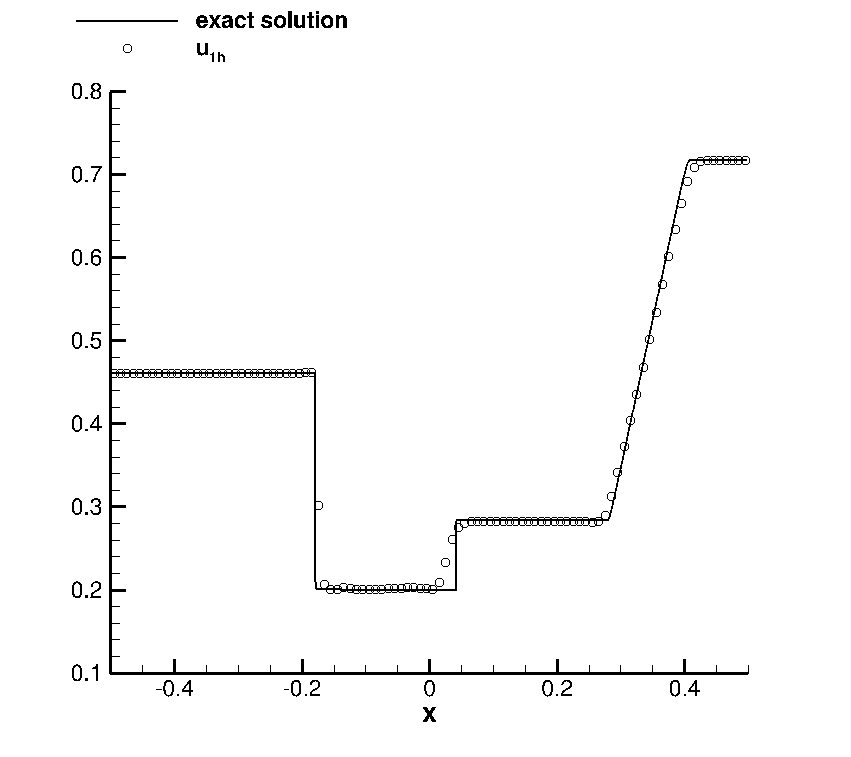   ,width=4cm}}
\subfigure{\epsfig{figure=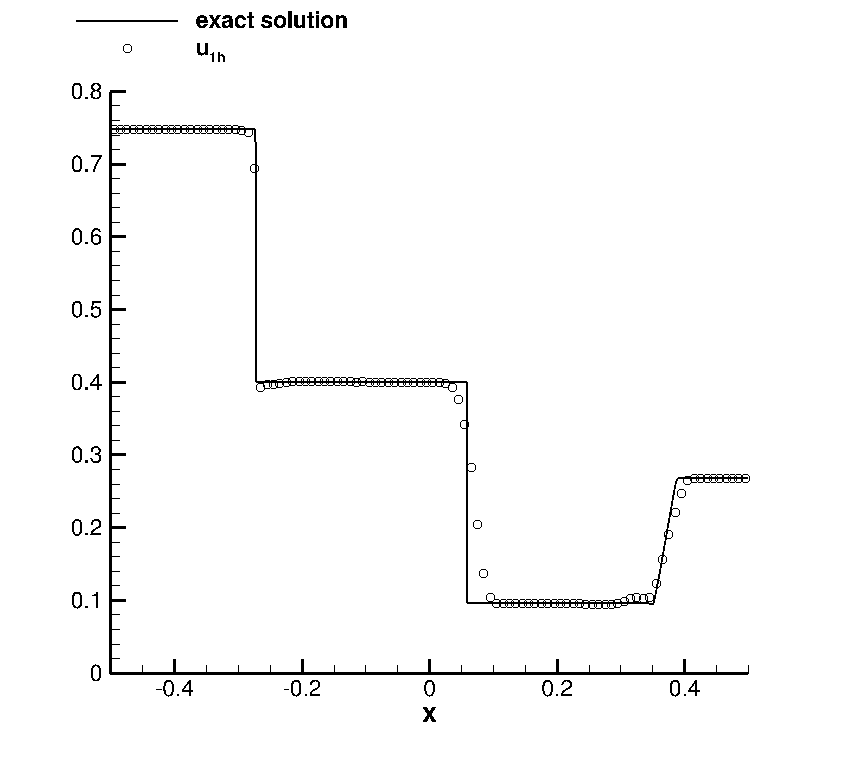   ,width=4cm}}
\subfigure{\epsfig{figure=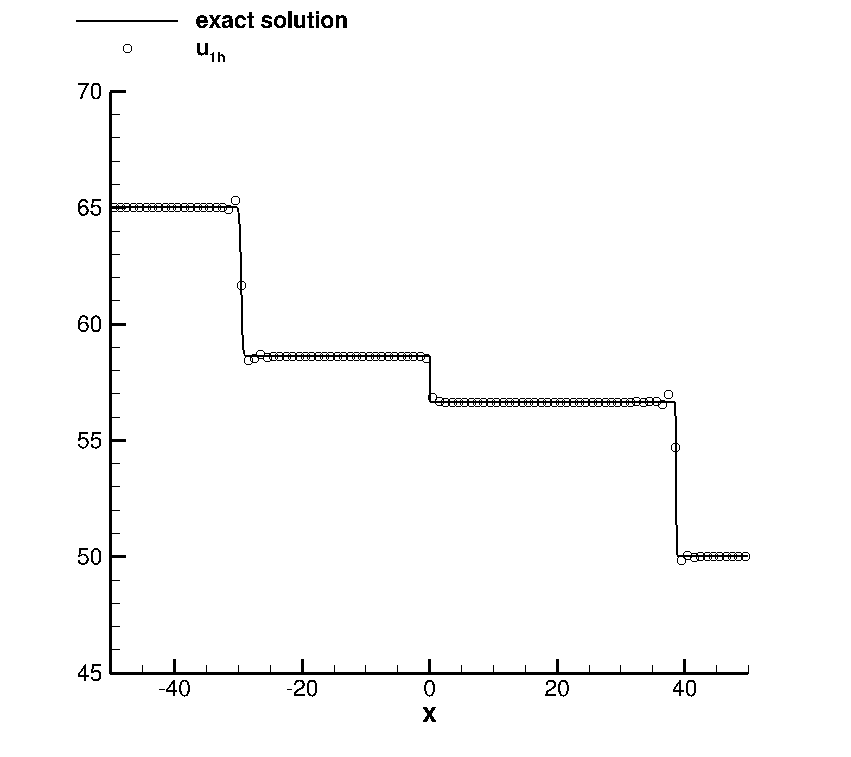   ,width=4cm}}\\
\subfigure{\begin{picture}(0,0) \put(-6,50){$\rho_2$} \end{picture}}
\subfigure{\epsfig{figure=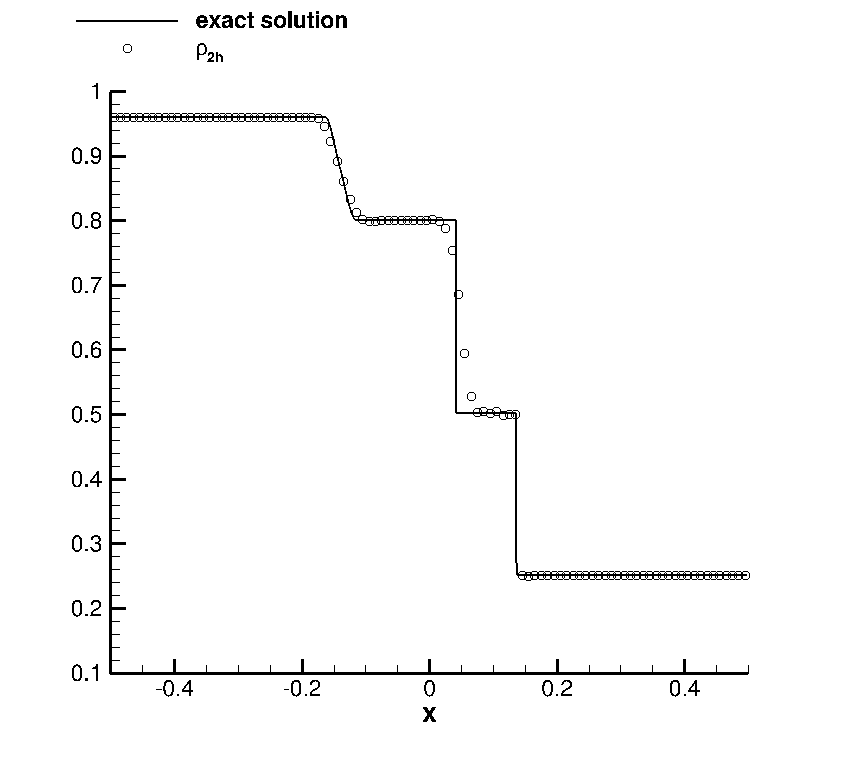 ,width=4cm}}
\subfigure{\epsfig{figure=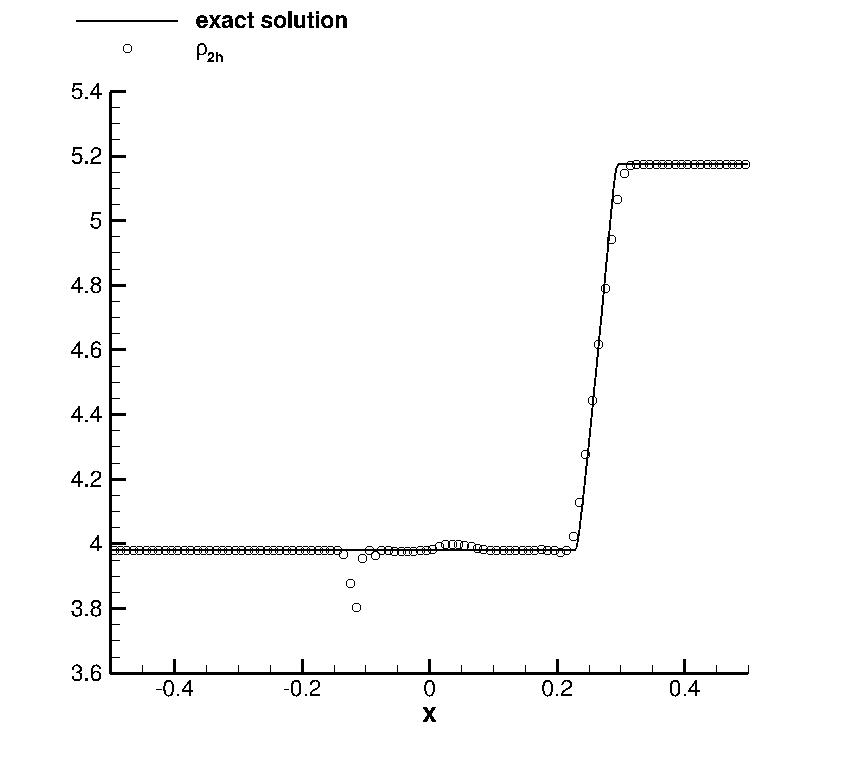 ,width=4cm}}
\subfigure{\epsfig{figure=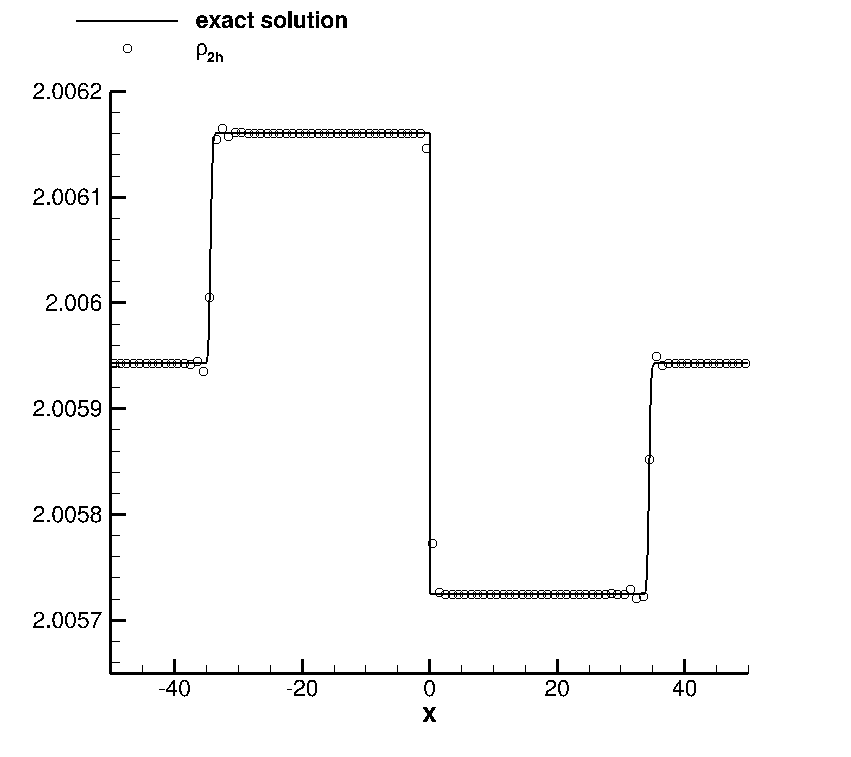 ,width=4cm}}\\
\subfigure{\begin{picture}(0,0) \put(-6,50){$u_2$} \end{picture}}
\setcounter{subfigure}{0}
\subfigure[RP1]{\epsfig{figure=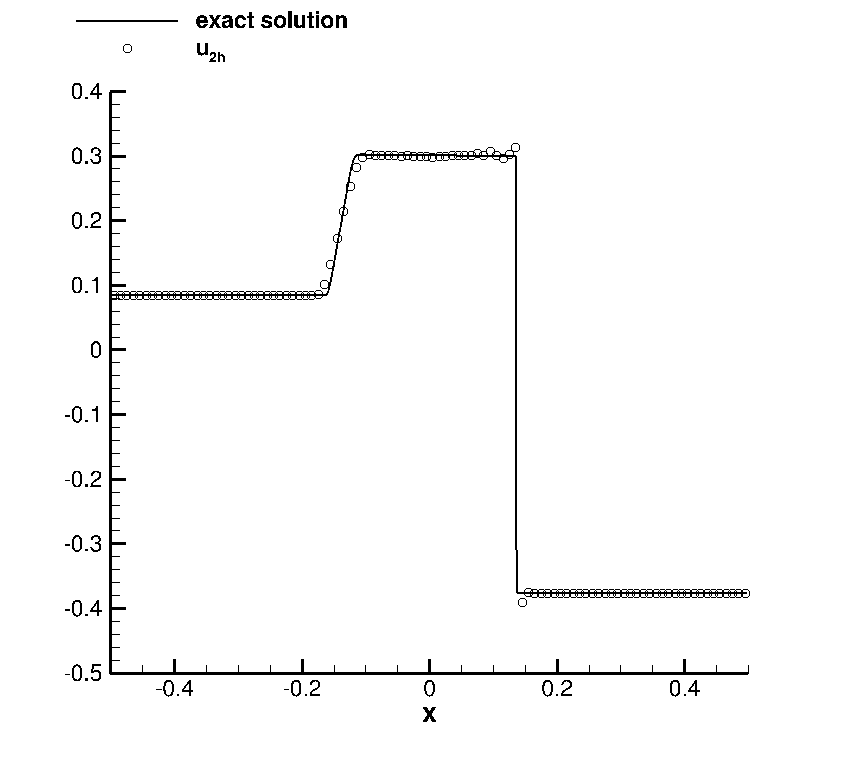,width=4cm}}
\subfigure[RP2]{\epsfig{figure=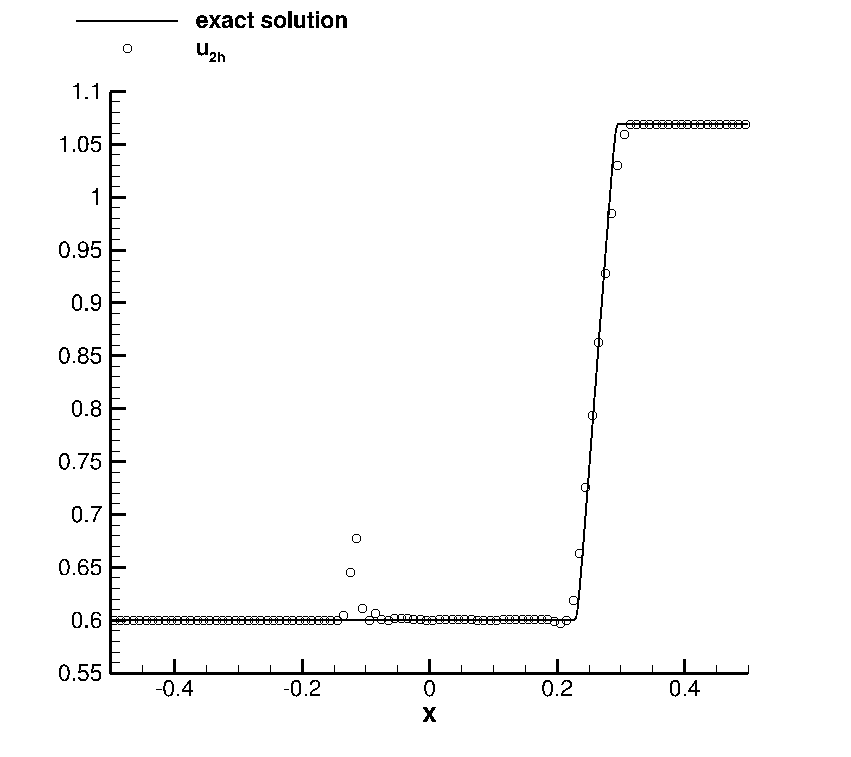,width=4cm}}
\subfigure[RP3]{\epsfig{figure=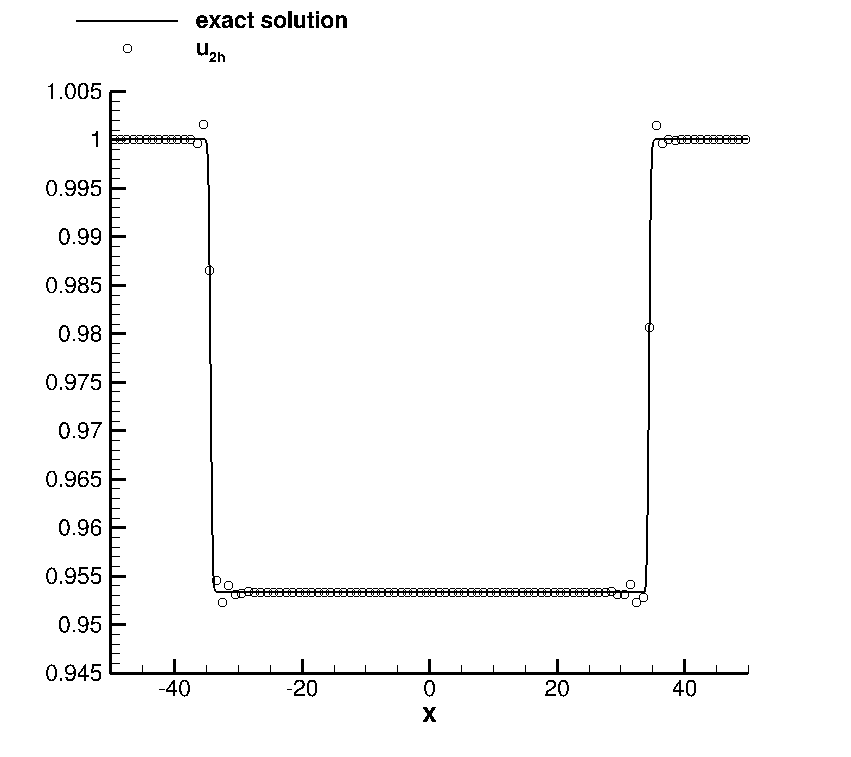,width=4cm}}
\caption{Isentropic Baer-Nunziato model: Riemann problems discretized with a polynomial degree $p=3$, $N=100$ cells and entropy stable scheme. From top to bottom: $\alpha_1$, $\rho_1$, $u_1$, $\rho_2$, and $u_2$.}
\label{fig:solution_PR_BN}
\end{bigcenter}
\end{figure}

%
%
\section{Concluding remarks}\label{sec:conclusions}

In this work, we introduce a general framework for the design of entropy stable DGSEM for the discretization of nonlinear hyperbolic systems in nonconservative form. The framework relies on the use of SBP operators and two-point entropy conservative fluctuation fluxes \cite{castro_etal_13} to evaluate the integral over discretization elements, thus removing its contribution to the global entropy production within the element, together with entropy stable fluxes at element interfaces. The framework may be seen as a generalization of the work on entropy stable DGSEM for conservation laws introduced in \cite{chen_shu_17}. In particular, the generalizations to multiple space dimensions with quadrangles, hexahedra, or simplex elements; the use of bound-preserving or TVD limiters; and the disretization of viscous terms will keep the entropy inequality as shown in \cite{chen_shu_17}.

Applications show that the methods proves to be robust, stable and entropy satisfying for the high-order discretization of two-phase flow models: a $2\times2$ system with a nonconservative product associated to a LD field and the isentropic Baer-Nunziato model. Future work will concern the improvement of the limiter to preserve contact discontinuities, the analysis of the well-balanced property, and the extension of the method to the Baer-Nunziato model with general equations of states including the transport equations for partial energies \cite{baer_nunziato86}.


%
%
%

\end{document}